\documentclass[twoside,11pt,reqno,letter]{amsart}

\usepackage{amsmath,amssymb,amscd,mathrsfs,epic,latexsym,tikz,mathrsfs,cite,hyperref}
\usepackage{wasysym,enumerate}
\usetikzlibrary{decorations.markings}

\numberwithin{equation}{section}
\newtheorem{Proposition}[equation]{Proposition}
\newtheorem{Lemma}[equation]{Lemma}
\newtheorem{Theorem}[equation]{Theorem}
\newtheorem{Corollary}[equation]{Corollary}
\theoremstyle{definition}  %% makes all of the theorem environments which follow appear in \rm
\newtheorem{Definition}[equation]{Definition}
\newtheorem{Remark}[equation]{Remark}

\let\<\langle
\let\>\rangle

\newcommand\Comment[2][\relax]{\space\par\medskip\noindent%
   \fbox{\begin{minipage}{\textwidth}\textbf{Comment\ifx\relax#1\else---#1\fi}\newline%
        #2\end{minipage}}\medskip
}

\DeclareMathOperator\ldom {\triangleleft}
\DeclareMathOperator\ledom{\trianglelefteq}
\DeclareMathOperator\gdom {\triangleright}
\DeclareMathOperator\gedom{\trianglerighteq}

\def\bi{\text{\boldmath$i$}}
\def\bj{\text{\boldmath$j$}}

\def\ba{\text{\boldmath$a$}}

\def\b1{\text{\boldmath$1$}}
\def\pmod#1{\text{ }(\text{\rm mod } #1)\,}

\newcommand{\words}{\langle I \rangle}
\newcommand{\Hom}{\operatorname{Hom}}

\def\qdim{{\operatorname{dim}_q}\,}\newcommand{\Gc}{\operatorname{Gc}}
\newcommand{\Arm}{\operatorname{Arm}}
\newcommand{\Leg}{\operatorname{Leg}}
\newcommand{\Shuf}{\operatorname{Shuf}}
\newcommand{\Sh}{\operatorname{Sh}}

\def\sgn{\mathtt{sgn}}
\newcommand{\res}{\operatorname{res}}

\newcommand{\St}{\operatorname{St}}

\newcommand{\cont}{\operatorname{cont}}

\newcommand{\Stab}{\operatorname{Stab}}

\newcommand{\F}{\mathbb{F}}
\newcommand{\Z}{\mathbb{Z}}
\newcommand{\N}{\mathbb{N}}

\def\phi{{\varphi}}

\newcommand{\ga}{\gamma}

\newcommand{\la}{\lambda}
\newcommand{\La}{\Lambda}
\newcommand{\al}{\alpha}

\def\Si{\mathfrak{S}}
\newcommand{\si}{\sigma}

\newcommand{\de}{\delta}

\def\Belt{\mathbf{B}}

\newcommand{\Ann}{{\mathrm {Ann}}}
\newcommand{\C}{{\mathbb C}}

\newcommand{\D}{{\mathscr D}}

\newcommand\Par{\mathscr P}

\def\T{{\mathtt T}}
\def\U{{\mathtt U}}
\def\Stab{{\mathtt S}}

\def\nslash{\:\notslash\:}

\def\height{{\operatorname{ht}}}

\def\wt{{\operatorname{wt}}}

\renewcommand\O{\mathcal O}

\newcommand\SetBox[2][35mm]{\Big\{\vcenter{\hsize#1\centering#2}\Big\}}

{\catcode`\|=\active
  \gdef\set#1{\mathinner{\lbrace\,{\mathcode`\|"8000%
  \let|\midvert #1}\,\rbrace}}
}
\def\midvert{\egroup\mid\bgroup}

\colorlet{darkgreen}{green!50!black}
\tikzset{dots/.style={very thick,loosely dotted},
         greendot/.style={fill,circle,color=darkgreen,inner sep=1.5pt,outer sep=0}
}
\def\greendot(#1,#2){\node[greendot] at(#1,#2){}}

\newenvironment{braid}{% sets defaults for the braid diagrams
  \begin{tikzpicture}[baseline=6mm,blue,line width=1pt, scale=0.4,
                      draw/.append style={rounded corners},
                      every node/.append style={font=\tiny}]%
  }{\end{tikzpicture}
}

\def\Grid(#1,#2){%  draws a coordinate grid inside a braid diagram
  \draw[very thin,gray,step=2mm] (0,0)grid(#1,#2);
  \draw[very thin,darkgreen,step=10mm] (0,0)grid(#1,#2);
}

\newcommand\Tableau[2][\relax]{
  \begin{tikzpicture}[scale=0.5,draw/.append style={thick,black}]
    \ifx\relax#1\relax%
    \else % shade the boxes in #1
      \foreach\box in {#1} { \filldraw[blue!30]\box+(-.5,-.5)rectangle++(.5,.5); }
    \fi
    \newcount\row\newcount\col
    \row=0
    \foreach \Row in {#2} {
       \col=1
       \foreach\k in \Row {
          \draw(\the\col,\the\row)+(-.5,-.5)rectangle++(.5,.5);
          \draw(\the\col,\the\row)node{\k};
          \global\advance\col by 1
       }
       \global\advance\row by -1
    }
  \end{tikzpicture}
}

\newcommand\YoungDiagram[2][\relax]{
  \begin{tikzpicture}[scale=0.5,draw/.append style={thick,black}]
    \ifx\relax#1\relax%
    \else % shade the boxes in #1
    \foreach\box in {#1} {
      \filldraw[blue!30]\box rectangle ++(1,1);
    }
    \fi
    \newcount\row
    \row=0
    \foreach \col in {#2} {
       \draw(1,\the\row)grid ++(\col,1);
       \global\advance\row by -1
    }
  \end{tikzpicture}
}

\newcommand\LeftBrac[2][\relax]{
  \begin{tikzpicture}[scale=0.5,draw/.append style={thick,black}]
    \ifx\relax#1\relax%
    \else % shade the boxes in #1
    \foreach\box in {#1} {
      \filldraw[blue!30]\box rectangle ++(1,1);
    }
    \fi
    \newcount\row
    \row=0
    \foreach \col in {#2} {
       \draw(1,\the\row)grid ++(\col,1);
       \global\advance\row by -1
    }
  \end{tikzpicture}
}

\begin{document}

\title[Specht module homomorphisms]{Homomorphisms from an arbitrary Specht module to one corresponding to a hook}

\author{\sc Joseph W. Loubert}
\address{Department of Mathematics\\ University of Oregon\\
Eugene\\ OR~97403, USA}
\email{loubert@uoregon.edu}
\begin{abstract}
We use the presentation of the Specht modules given by Kleshchev-Mathas-Ram to derive results about Specht modules. In particular, we determine all homomorphisms from an arbitrary Specht module to a fixed Specht module corresponding to any hook partition. Along the way, we give a complete description of the action of the standard KLR generators on the hook Specht module. This work generalizes a result of James.
\end{abstract}

\maketitle

\section{Introduction}
Let $H_d$ be an Iwahori-Hecke algebra for the symmetric group $\Si_d$ with deformation parameter $q$, over a field $F$. Define $e$ to be the smallest integer so that $1 + q + \dots + q^{e-1} = 0$, setting $e=0$ if no such value exists.
To every partition $\la$ of $d$, there is a corresponding \emph{Specht module} $S^\la$ over $H_d$.
Brundan and Kleshchev show in \cite{BK} that $H_d$ is isomorphic to a certain algebra $R^{\La_0}_d$ known as a cyclotomic KLR (Khovanov-Lauda-Rouquier) algebra of type $A^{(1)}_{e-1}$ when $e\neq 0$, and type $A_\infty$ when $e = 0$.
The Specht modules are described in \cite{BKW} as modules over the cyclotomic KLR algebras.
This result is extended by Kleshchev, Mathas, and Ram in \cite{KMR}, in which Specht modules for the (full) KLR algebras $R_d$ are explicitly defined in terms of generators and relations. 
%In this paper we will be working with the KLR algebras $R_d$, which by \cite{BK}, means most our results also apply to the Iwahori-Hecke algebras of the symmetric group.
The purpose of this paper is to use this presentation to completely determine $\Hom_{R_d}(S^\mu, S^\la)$ when $\mu$ is an arbitrary partition and $\la$ is a hook.
Of course, when $e=0$ there are no nontrivial homomorphisms. Furthermore, our methods do not apply when $e=2$, so we make the assumption that $e \geq 3$.

%, say $\la = (d-k,1^{(k)})$ ; see Theorem~\ref{thm:mainHom}. We note here that our results only hold in characteristic at least 3, and that our Specht modules are dual to the classical versions of James~\cite{J}. Some special cases of the main theorem are notable. First of all, when $\la=(d)$ we recover (a dualized version of) the classical result of James \cite[Theorem 24.4]{J}. When $\la = (d-1,1)$, we obtain the following new result. Let $\nu_p$ and $\ell_p$ be as in Definition~\ref{def:pstuff}.

To state the main theorem, we need some notation. The element $z^\mu \in S^\mu$ is the standard cyclic generator of weight $\bi^\mu$, see Definition~\ref{DSpecht}.
When the $\bi^\mu$-weight space of $S^\la$ is nonempty, it will turn out that its top degree component is one-dimensional. In this case we take $[\si_\mu] \in S^\la$ to be any non-zero vector of that component.
For any weakly decreasing sequence of positive integers $\ba = (a_1, \dots, a_N)$ we define its \emph{Garnir content} $\Gc(\ba) \in F$ by
\begin{equation*}%\label{DGc}
    %\Gc(a_1, \dots, a_N) =  & \gcd\left(\binom{a_1}{1}\right., \binom{a_1}{2}, \dots, \binom{a_1}{a_2 - 1}, \\
    %         &\qquad  \binom{a_2}{1}, \binom{a_2}{2}, \dots, \binom{a_2}{a_3 - 1}, \dots, \\
    %         &\qquad  \binom{a_{N-1}}{1}, \binom{a_{N-1}}{2}, \dots, \left. \binom{a_{N-1}}{a_N - 1} \right),
    \Gc(\ba) = \gcd\left\{ \binom{a_i}{k} \mid 1 \leq k \leq a_{i+1}-1,\ 1 \leq i \leq N-1\right\}
\end{equation*}
with the convention that $\gcd(\emptyset) = 0$. %Thus, for example, $\Gc(a_1) = 0$.

\vspace{2mm}
\noindent
{\bf Main Theorem.}
{\em
  Let $\mu$ be an arbitrary partition and $\la = (d-k, 1^k)$ for $k \geq 0$. Then $\Hom_{R_d}(S^\mu, S^\la)$ is at most one-dimensional, spanned by a map satisfying $z^\mu \mapsto [\si_\mu]$, and $\dim \Hom_{R_d}(S^\mu, S^\la) = 1$ if and only if one of the following conditions holds:
  \begin{enumerate}
    \item there exist $n \in \{1, \dots, k+1\}$, $\ba \in (\Z_{>0})^n$, and $0 \leq m < e$ such that $\Gc(\ba) = 0$ and
$$\mu = (a_1 e, \dots, a_{n-1} e, a_n e - m, 1^{k-n+1}),$$
    \item $e$ divides $d$, there exist $n \in \{1, \dots, k\}$, $\ba \in (\Z_{>0})^n$, and $0 \leq m < e$ such that $\Gc(\ba) = 0$ and
$$\mu = (a_1 e, \dots, a_{n-1} e, a_n e - m, 1^{k-n+2}),\ \text{or}$$
    \item there exist $n > k+1$, $\ba \in (\Z_{>0})^n$ and $0 \leq m < e$ such that $\Gc(\ba) = 0$ and
$$\mu = (a_1 e, \dots, a_k e, a_{k+1}e - 1, \dots, a_{n-1} e - 1, a_n e - 1 - m).$$
  \end{enumerate}
}

This theorem generalizes James \cite[Theorem 24.4]{J}, which corresponds to $k=0$. Our Specht modules are the dual of the Specht modules defined in \cite{J}. We note here that the main theorem also allows us to determine all homomorphisms between Specht modules when the source is a hook, as follows; see \cite{KMR} for details. For any partition $\nu$ we define $\nu'$ to be the conjugate partition and $S_\nu$ to be the dual of $S^\nu$. There is an automorphism $\sgn$ of $R_d$ such that $S_\nu \cong (S^{\nu'})^\sgn$, and so
\begin{align*}
\Hom_{R_d}(S^\mu,S^\la) &\cong \Hom_{R_d}(S_\la, S_\mu) \cong \Hom_{R_d}((S^{\la'})^\sgn, (S^{\mu'})^\sgn) \\
&\cong \Hom_{R_d}(S^{\la'}, S^{\mu'}).
\end{align*}
Clearly, $\la$ is a hook if and only if $\la'$ is.

Section 2 introduces the notation to be used throughout the paper. In addition, we collect a few elementary facts which will be necessary in later sections.
Section 3 reviews the definitions of the affine and cyclotomic KLR algebras, and their universal Specht modules as introduced in \cite{KMR}.
Section 4 provides a study of the structure of the Specht modules $S^\la$, where $\la$ is a hook. Specifically, there is a basis for $S^\la$ given by standard tableaux on $\la$, and in Theorem~\ref{thm:KLRonBasis} we determine how the generators of the KLR algebra act on this basis.
In section 5, we look at homomorphisms from an arbitrary Specht module $S^\mu$ to $S^\la$, where $\la$ is a hook.
In section 6, we consider some examples.

\section{Preliminaries}
\subsection{Lie theoretic notation}
We collect some notation that will be used in the sequel; the interested reader is refered to~\cite{KMR} for more details. Let $e \in \{3, 4, \dots\}$ and $I := \Z/e\Z.$ Let $\Gamma$ be the quiver with vertex set~$I$, with a directed edge from~$i$ to~$j$ if $j = i - 1$. Thus~$\Gamma$ is a quiver of type $A^{(1)}_{e-1}$. We denote the simple roots by $\{\al_i\mid i\in I\},$ and define $Q_+ := \bigoplus_{i \in I} \Z_{\geq 0} \alpha_i$ to be the positive part of the root lattice. For $\alpha \in Q_+$ let $\height(\alpha)$ be the {\em height of~$\al$}. That is, $\height(\al)$ is the sum of the coefficients when $\al$ is expanded in terms of the $\alpha_i$'s.

Let~$\Si_d$ be the symmetric group on~$d$ letters and let $s_r = (r, r+1)$, for $1\leq r < d$, be the simple transpositions of~$\Si_d$. Then~$\Si_d$ acts on the left on the set~$I^d$ by place permutations.
If $\bi = (i_1, \dots , i_d) \in I^d$ then its {\em weight} is
$|\bi| := \alpha_{i_1} + \cdots + \alpha_{i_d} \in Q_+$.  The $\Si_d$-orbits on $I^d$ are the sets 
\[
  \words_\alpha := \{\bi \in I^d\mid \al=|\bi|\} 
\]
parametrized by all $\alpha \in Q_+$ of height $d$.

\subsection{Partitions}\label{SSPar}
Let $\Par_d$ be the set of all partitions of $d$ and put $\Par:=\bigsqcup_{d\geq 0}\Par_d$. 
The {\em Young diagram} of the partition $\mu\in \Par$ is 
$$
\{(a,b)\in\Z_{>0}\times\Z_{>0}\mid 1\leq b\leq \mu_a\}.
$$
The elements of this set
are the {\em nodes of $\mu$}. More generally, a {\em node} is any element of $\Z_{>0}\times\Z_{>0}$. 

To each node $A=(a,b)$ 
we associate its {\em residue}, which is the following element of $I=\Z/e\Z$: 
\begin{equation}\label{ERes}
\res A=(b-a)\pmod{e}.
\end{equation}
An {\em $i$-node} is a node of residue $i$.
Define the {\em residue content of $\mu$} to be
\begin{equation}\label{EContent}
\cont(\mu):=\sum_{A\in\mu}\al_{\res A} \in Q_+.
%\sum_{i\in I}m_i\al_i.
\end{equation}
Denote
$$
\Par_\al:=\{\mu\in\Par\mid \cont(\mu)=\al\}\qquad(\al\in Q_+).
$$

A node $A\in\mu$ is a {\em removable node (of~$\mu$)}\, if $\mu\setminus \{A\}$ is (the diagram of) a partition. A node $B\not\in\mu$ is an {\em addable node (for~$\mu$)}\, if $\mu\cup \{B\}$ is a partition. We use the notation
$$
\mu_A:=\mu\setminus \{A\},\qquad \mu^B:=\mu\cup\{B\}.
$$

\subsection{Tableaux}\label{SS:tableaux}
Let $\mu\in\Par_d$. 
A {\em $\mu$-tableau} $\T$ is obtained from the diagram of $\mu$ by 
inserting the integers $1,\dots,d$ into the nodes, allowing no repeats. 
If the node $A=(a,b)\in\mu$ is occupied by the integer $r$ in $\T$ then we write $r=\T(a,b)$ and set
$\res_\T(r)=\res A$. The {\em residue sequence} of~$\T$ is
\begin{equation}\label{EResSeq}
\bi(\T)=(i_1,\dots,i_d)\in I^d,
\end{equation}
where $i_r=\res_\T(r)$ is the residue of the node occupied by 
$r$ in $\T$ ($1\leq r\leq d$). 

%\Comment[Andrew]{I added $\bi^\kappa(\T)$ here and similar changes to Lemma~8.4 below.}

A $\mu$-tableau $\T$ is {\em row-strict} (resp. {\em column-strict}) if its entries increase from left to right (resp. from top to bottom) along the rows (resp. columns) of each component of $\T$. 
A $\mu$-tableau $\T$ is {\em standard} if it is row- and column-strict. 
 Let $\St(\mu)$ be the set of standard $\mu$-tableaux.

\iffalse
Let $\T$ be a $\mu$-tableau and suppose that $1\leq r\neq s\leq d$ and that $r=\T(a_1,b_1,m_1)$ and that
$s=\T(a_2,b_2,m_2)$. We write $r\nearrow_\T s$ if $m_1=m_2$, $a_1>a_2$, and $b_1<b_2$; informally, $r$ and $s$
are in the same component and~$s$ is strictly to the north-east of $r$ within that component.  The symbols $\rightarrow_\T,\searrow_\T,\downarrow_\T$ have the  similar obvious meanings. For example,
$r\downarrow_\T s$ means that $r$ and $s$ are located in the same column of the same component of~$\T$ and
that~$s$ is in a strictly lower row of~$\T$ than~$r$. 
\fi

%\subsection{Degree of a standard tableau}\label{SSDeg}
Let $\mu\in\Par$, $i\in I$, and $A$ be a removable $i$-node
%and $B$ be an addable $i$-node 
of $\mu$. We set
\begin{equation}\label{EDMUA}
d_A(\mu)=\#\SetBox{addable $i$-nodes of $\mu$\\[-4pt] strictly below $A$}
                -\#\SetBox[38mm]{removable $i$-nodes of $\mu$\\[-4pt] strictly below $A$}.
\end{equation}
%Also, for $i\in I$, define
%\vspace{.5 mm}
%\begin{equation}\label{EDKWeight}
%d_i(\mu):=\#\{\text{addable $i$-nodes of $\mu$}\}-\#\{\text{removable $i$-nodes of $\mu$}\}.
%\end{equation} 
%It is easy to see \cite[Lemma 3.11]{BKW} that:
%\begin{equation}\label{EDI}
%d_i(\mu)=(\La-\al,\al_i)\qquad(\mu \in \Par_\al).
%\end{equation}

Given $\mu \in \Par_d$ and $\T \in \St(\mu)$, the {\em degree} of $\T$ is defined in
\cite[section~3.5]{BKW} inductively as follows. If $d=0$, then $\T$ is the empty tableau $\emptyset$, and
we set $\deg(\T):=0$.  Otherwise, let $A$ be the node occupied by $d$ in $\T$. Let $\T_{<d}\in\St(\mu_A)$
be the tableau obtained by removing this node and set
\begin{equation}\label{EDegTab}
\deg(\T):=d_A(\mu)+\deg(\T_{<d}).
\end{equation}

The group $\Si_d$ acts on the set of $\mu$-tableaux from the left by acting on the entries of the tableaux.
Let $\T^\mu$ be the $\mu$-tableau in which the numbers $1,2,\dots,d$ appear in order from left to right along the successive rows,
working from top row to bottom row.

Set 
\begin{equation}\label{EBIMu}
\bi^\mu:=\bi(\T^\mu).
\end{equation} 
For each $\mu$-tableau $\T$ define permutations $w^\T \in \Si_d$ by the equation
\begin{equation}\label{EWT}
w^\T  \T^\mu=\T.
\end{equation}

\subsection{Binomial coefficients}\label{ss:binom}

In this section, we state some elementary theorems about binomial coefficients that will be useful later. Let $p$ be a fixed prime.

\begin{Definition}\label{def:pstuff}
  For $n \in \Z_{>0}$ we define
  \begin{enumerate}
    \item $\nu_p(n) = \max \{ i \mid p^i \text{ divides } n \}$
    \item $\ell_p(n) = \min \{ i \mid p^i > n \}$.
  \end{enumerate}
  We also set $\ell_p(0) = -\infty$.
\end{Definition}

The following can be easily derived from \cite[Lemma 22.4]{J}.
\begin{Lemma}\label{cor:pBinom}
  For any $a, b \in \Z_{>0}$, one has
  \[
    p \left| \binom{a}{k}\right. \textup{ for } k = 1,\dots, b \qquad \Leftrightarrow \qquad \nu_p(a) \geq \ell_p(b).
  \]
\end{Lemma}
Recall the definition of $\Gc(a_1, \dots, a_N)$ from the introduction. Then
\begin{Corollary}\label{cor:pGc}
  We have $p \ |\ \Gc(a_1, \dots, a_N)$ if and only if 
$$\nu_p(a_i) \geq \ell_p(a_{i+1}-1)  \text{ for } 1 \leq i \leq N-1.$$
\end{Corollary}

\subsection{Shuffles}\label{ss:shuf}

  In this section, we fix $e \in \{0, 3, 4, 5, \dots\}$ and $I := \Z/e\Z$. Given $a, b \in \Z_{\geq 0}$ we define 
$$\Sh(a,b) := \{\sigma \in \Si_{a+b}\ |\ \sigma(1) < \dots < \sigma(a) \text{ and }\sigma(a+1) < \dots < \sigma(a+b)\}.$$
% These are the minimal length coset representatives of $\Si_a \times \Si_b$ in $\Si_{a+b}$.
For $\bi \in I^a, \bi' \in I^b$ write $\bi \bi' \in I^{a+b}$ for their concatenation, and define
\begin{align*}
\Shuf(\bi, \bi') &:= \{\sigma \cdot \bi \bi'\ |\ \sigma \in \Sh(a,b)\},\\
%and call the elements of this set shuffles of~$\bi$ with~$\bi'$.
%Given $\bj \in \Shuf(\bi, \bi')$, define
\Sh(\bj; \bi, \bi') &:= \{ \sigma \in \Sh(a,b)\ |\ \sigma \cdot \bi \bi' = \bj\},\text{ and}\\
%For a word $\bi \in I^d$, define
H(\bi) &:= \langle s_m\ |\ s_m \bi = \bi \rangle < \Si_a.
\end{align*}
For $i \in I$, define the elements of $I^a$
\begin{align}
  S^+(i, a) &:= (i, i + 1, \dots, i + a - 1),\label{def:IncSeg}\\
  S^-(i, a) &:= (i, i - 1, \dots, i - a + 1).\label{def:DecSeg}
\end{align}

Fix $j, k \in I$, $a, b \in \Z_{\geq 0}$, and write $S^+ := S^+(j, a)$ and $S^- := S^-(k, b)$. If $\bi \in \Shuf(S^+, S^-)$ then~$\bi$ cannot contain three or more equal adjacent indices. In this case $H(\bi)$ is an elementary 2-subgroup of $\Si_{a+b}$.

\begin{Proposition}\label{prop:Shortest}
  For every   $\bi \in \Shuf(S^+, S^-)$, there is a unique element $\si_\bi \in \Sh(\bi; S^+, S^-)$ of minimal length. Furthermore, $\Sh(\bi; S^+, S^-) = \{h \sigma_\bi\ |\ h \in H(\bi) \}$.
\end{Proposition}

\begin{proof}
Write $S^+ = (j_1, \dots, j_a)$ and $S^- = (k_1, \dots, k_b)$, and for every integer $d \geq 1$ introduce the following notation. For $\si \in \Si_d$, define $E\si \in \Si_{d-1}$ by
$$E\si(r) = \begin{cases} \si(r),& \text{ for } r=1,\dots,\si^{-1}(d)-1;\\
  \si(r+1),& \text{ for } r=\si^{-1}(d),\dots,d-1. \end{cases}$$
 Similarly, for $\bi = (i_1, \dots, i_d) \in I^d$, define $E\bi = (i_1, \dots, i_{d-1})$.

The proposition will be proved by induction on $a+b$. If $a+b = 1$ or $2$, or if $a = 0$ or $b = 0$ the claim follows immediately.
Suppose $a+b \geq 3$, with $a, b \geq 1$, and assume the induction hypothesis.
Fix $\bi \in \Shuf(S^+, S^-)$.
We distinguish several cases based on the values of $i_{a+b}, i_{a+b-1}, j_a,$ and $k_b.$

Case 1: $i_{a+b} = j_a$ and either $j_a \neq k_b$ or $i_{a+b-1}=i_{a+b}-1$. Given $\si \in \Sh(\bi; S^+, S^-)$, we must have $\si(a) = a+b$. Then the map 
$$\Sh(\bi; S^+, S^-) \to \Sh(E\bi; ES^+,S^-),\ \si \mapsto E\si$$
is a bijection, and moreover $\ell(E\si) = \ell(\si) - b$ and $H(\bi) \cong H(E\bi)$. We define $\si_\bi$ by $E\si_\bi = \si_{E\bi}$, which completes the induction in this case.

Case 2: $i_{a+b} = k_b$ and either $j_a \neq k_b$ or $i_{a+b-1}=i_{a+b}+1$. Here the map
$$\Sh(\bi; S^+, S^-) \to \Sh(E\bi; S^+, ES^-),\ \si \to E\si$$
gives a bijection as above.

Case 3: $i_{a+b} = i_{a+b-1} = j_a = k_b$. We must either have $\si(a) = a+b$ and $\si(a+b) = a+b-1$ or $\si(a)=a+b-1$ and $\si(a+b) = a+b$. Thus, we get a well-defined map
$$\Sh(\bi; S^+, S^-) \to \Sh(EE\bi; ES^+, ES^-),\ \si \mapsto EE\si$$
which is two-to-one, with $s_{a+b-1} \si$ and $\si$ having the same image. This proves the claim, since $H(\bi) \cong H(EE\bi) \times \<s_{a+b-1}\>$.
\end{proof}

\section{KLR algebras and universal Specht modules}

\subsection{KLR algebras}
Let $\O$ be a commutative ring with identity and $\al\in Q_+$.  There is a unital $\O$-algebra $R_\al=R_\al(\O)$ called a {\em Khovanov-Lauda-Rouquier (KLR) algebra}, first defined in \cite{KL1,KL2,R}. We follow the notations and conventions of \cite{KMR}, so $R_\al$ is generated by elements
\begin{equation}\label{EKLGens}
\{e(\bi)\:|\: \bi\in \words_\al\}\cup\{y_1,\dots,y_{d}\}\cup\{\psi_1, \dots,\psi_{d-1}\},
\end{equation}
subject to some explicit relations.

The algebras $R_\al$ have $\Z$-gradings determined by setting 
$e(\bi)$ to be of degree 0,
$y_r$ of degree $2$, and
$\psi_r e(\bi)$ of degree $-a_{i_r,i_{r+1}}$
for all $r$ and $\bi \in \words_\al$.

For a graded vector space $V=\oplus_{n\in \Z} V_n$, with finite dimensional graded components its {\em graded dimension} is $\qdim \, V:=\sum_{n \in \Z}  (\dim V_n)q^n\in\Z[[q,q^{-1}]]$.

%Note that $R_\al(\Z)\otimes_\Z \O\cong R_\al(\O)$ and $R_\al^\La(\Z)\otimes_\Z \O\cong R_\al^\La(\O)$.

\subsection{Universal Specht modules}\label{SecSpecht}
Fix $\mu \in \Par_\al$. In this section we define the Specht module $S^\mu$ over $R_\al$, mostly following \cite{KMR}. A node $A=(x,y) \in \mu$ is called a \emph{Garnir node} of $\mu$ if also $(x,y+1) \in \mu$. We define the \emph{Garnir belt} of $A$ to be the set $\Belt^A$ of nodes of $\mu$ containing $A$ and all nodes directly to the right of $A$, along with the node directly below $A$ and all nodes directly to the left of this node. Explicitly,
$$\Belt^A = \{(x, z) \in \mu \mid y \leq z \leq \mu_x\} \cup \{(x+1, z) \in \mu \mid 1 \leq z \leq y\}.$$
We define an \emph{$A$-brick} to be a set of $e$ successive nodes in the same row
$$\{(w,z), (w,z+1),\dots,(w,z+e-1)\} \subseteq \Belt^A$$
such that $\res(z,w) = \res(A)$. Let $k \geq 0$ be the number of bricks in $\Belt^A$. We label the bricks
$$B^A_1, B^A_2, \dots, B^A_k$$
going from left to right along row $x$, and then from left to right along row $x+1$. Define $C^A$ to be the set of nodes in row $x$ of $\Belt^A$ not contained in any $A$-brick, and $D^A$ to be the set of nodes in row $x+1$ of $\Belt^A$ not contained in any $A$-brick. 
%$$\Belt^A = B^A_1 \amalg B^A_2 \amalg \dots \amalg B^A_k \amalg C^A \amalg D^A.$$
Define $f$ to be the number of $A$-bricks in row $x$ of $\Belt^A$. %Thus $\Belt^A$ is partitioned as in the following picture.

%The \emph{$A$-Garnir tableau} is the $\mu$-tableau $\G^A$ formed by taking $\T^\mu$ and rearranging the values of the nodes in $\Belt^A$ in order, starting at the bottom left and ending at the top right.

Let $u, u+1, \dots, v$ be the values in $\Belt^A$ in the standard tableau $\T^\mu$. We obtain a new tableau $\T^A$ by placing the numbers $u, u+1 \dots, v$ from left to right in the following order: $D^A$, $B^A_1$, $B^A_2$, \dots, $B^A_k$, $C^A$. The rest of the numbers are placed in the same positions as in $T^\mu$. Define $\bi^A := \bi(\T^A)$.

Assume that $k > 0$, and let $n = \T^A(A)$. Define
$$
w^A_r = \prod_{z=n+re-e}^{n+re-1} (z, z+e) \in \Si_d \qquad (1\leq r < k).
$$
Informally, $w^A_r$ swaps the bricks $B^A_r$ and $B^A_{r+1}$. The elements $w^A_1, w^A_2, \dots, w^A_{k-1}$ are Coxeter generators of the group
$$\Si^A := \< w^A_1, w^A_2, \dots, w^A_{k-1} \> \cong \Si_k.$$
By convention, if $k=0$ we define $\Si^A$ to be the trivial group.

Recall that $f$ is the number of $A$-bricks in row $x$ of $\Belt^A$. Let $\D^A$ be the set of minimal length left coset representatives of $\Si_f \times \Si_{k-f}$ in $\Si^A \cong \Si_k$. Define the elements of $R_\al$
$$\si^A_r := \psi_{w^A_r} e(\bi^A) \quad\text{and}\quad \tau^A_r := (\si^A_r + 1)e(\bi^A).$$
Let $u \in \D^A$ with reduced expression $u = w^A_{r_1} \dots w^A_{r_a}$. Since every element of $\D^A$ is fully commutative, the element
$\tau^A_u := \tau^A_{r_1} \dots \tau^A_{r_a}$
does not depend upon this reduced expression. We now define the \emph{Garnir element} to be
$$g^A := \sum_{u \in \D^A} \tau^A_u \psi^{\T^A} \in R_\al.$$

\begin{Definition} \label{DSpecht}%{\rm \cite{}}%{\bf ()}
Let $\al\in Q_+$, $d=\height(\al)$, and $\mu\in\Par_\al$. Define the following left ideals of $R_\al$.
\begin{enumerate}
\item[{\rm (i)}] $J_1^\mu = \<e(\bj) - \de_{\bj,\bi^\mu} \mid \bj\in \words_\al\>$;
\item[{\rm (ii)}] $J_2^\mu = \<y_r \mid r=1,\dots,d\>$;
\item[{\rm (iii)}] $J_3^\mu = \<\psi_r \mid \text{ $r$ and $r+1$ appear in the same row of } \T^\mu\>$;
\item[{\rm (iv)}] $J_4^\mu = \< g^A \mid \text{ Garnir nodes } A \in \mu\>$. 
\end{enumerate}
Let $J^\mu = J_1^\mu + J_2^\mu + J_3^\mu + J_4^\mu$ and define the {\em universal graded Specht module}\, $S^\mu:=R_\al / J^\mu \<\deg(\T^\mu)\>$. Define $z^\mu := 1 + J^\mu \in S^\mu$.%, having degree $\deg(\T^\mu)$.
\end{Definition}

\section{Specht modules corresponding to hooks}

For the rest of the paper, we assume that $e \geq 3$. In this section we fix two integers $d \geq k \geq 0$. We set $\la := (d-k,1^k)$ and $\al := \cont(\la)$. We write $\Sh^\la$ for the image of $\Sh(d-k-1, k)$ under the embedding $\Si_{d-1} \to \Si_d$ determined by $s_i \mapsto s_{i+1}$. For $\si \in \Sh^\la$, we define $[\si] := \psi_{\si} z^\la \in S^\la$ of weight $\bi_\si := \si \bi^\la$.
Note that $\Sh^\la = \{\si \in \Si_d\ |\ \si \T^\la \textup{ is standard}\}$. Therefore, by \cite[Corollary 6.24]{KMR}, the set $\{[\si]\ |\ \si \in \Sh^\la\}$ is a basis of $S^\la$.%, which we call the \emph{standard basis}.

\begin{Remark}
  Recall that in order to define the element $\psi_\si \in R_\al$ we needed to fix a reduced decomposition for $\si$. However, every $\si \in \Sh^\la$ is fully commutative, and so the element $\psi_\si$ is independent of this choice. Thus $[\si] \in S^\la$ only depends on $\si$.
\end{Remark}

\begin{Definition}
  Given $\si \in \Sh^\la$, the strands in the braid diagram of $\si$ beginning at positions $2, 3, \dots, d - k$ are called \emph{arm strands}, and we define $\Arm(\si) = \{\si(2), \dots, \si(d - k)\}$. Similarly, the strands beginning at positions $d - k + 1, \dots, d$ are called \emph{leg strands}, and $\Leg(\si) = \{\si(d - k + 1), \dots, \si(d)\}$
\end{Definition}

%  We futhermore define $\Sh^\la(\bi) = \{\si \in \Sh^\la\ |\ \wt([\si]) = \bi\}$.

The following theorem tells us how the KLR generators act on the standard basis of $S^\la$.

\begin{Theorem}\label{thm:KLRonBasis}
  Let $\si \in \Sh^\la$ and write $\bi_\si = (i_1, \dots, i_d)$. Then
  \begin{enumerate}
    \item For $\bj \in \words_\al$, $e(\bj) [\si] = \de_{\bi_\si, \bj} [\si]$.
    \item For $1 \leq r \leq d$,
          \[ y_r [\si] = \begin{cases}
             -[s_r \si],& \textup{ if $i_r = i_{r+1}$, $r \in \Leg(\si)$, and $r+1 \in \Arm(\si)$}\\
             [s_{r-1} \si],& \textup{ if $i_{r-1} = i_r$, $r-1 \in \Leg(\si)$, and $r \in \Arm(\si)$}\\
             0,& \textup{ otherwise.}
          \end{cases} \]
    \item We have $\psi_1 [\si] = 0$, and for $2 \leq r \leq d-1$, 
          \[ \psi_r [\si] = \begin{cases}
             [s_r \si],& \textup{ if $r \in \Arm(\si)$ and $r+1 \in \Leg(\si)$, or if $i_r \nslash i_{r+1}$}\\
%             0,& \textup{ if $r \in \Leg(\si)$, $r+1 \in \Arm(\si)$, and $i_r = i_{r+1}$}\\
             [s_{r+1} s_r \si],& \textup{ if $r \in \Leg(\si)$, $r+1, r+2 \in \Arm(\si)$, and $i_r \to i_{r+1}$}\\
             -[s_r s_{r+1} \si],& \textup{ if $r-1, r \in \Leg(\si)$, $r+1 \in \Arm(\si)$, and $i_r \leftarrow i_{r+1}$}\\
             [s_r s_{r-1} \si],& \textup{ if $r-1 \in \Leg(\si)$, $r, r+1 \in \Arm(\si)$, and $i_r = i_{r-1}$}\\
             -[s_{r-1} s_r \si],& \textup{ if $r, r+1 \in \Leg(\si)$, $r+2 \in \Arm(\si)$, and $i_{r+1} = i_{r+2}$}\\
             0,& \textup{ otherwise.}
          \end{cases} \]
  \end{enumerate}
\end{Theorem}

\begin{proof}
  (i) Immediate from the definitions.

  (ii) We define $\Sh^\la(\bi) = \{\tau \in \Sh^\la\ |\ \wt([\tau]) = \bi\}$. By Proposition~\ref{prop:Shortest} there is a unique element $\tau \in \Sh^\la(\bi_\si)$ of minimal length, and $\si = h \tau$ for some $h \in H(\bi_\si)$. Choosing any reduced decomposition $h = s_{r_1} \dots s_{r_a}$, it follows that $[\si] = \psi_{r_1} \dots \psi_{r_a} [\si_\bi]$. Moreover $\deg([\tau])-\deg([\si]) = 2a$, and so $[\tau]$ is the unique vector of weight $\bi_\si$ with largest degree. In particular, $y_s [\tau] = 0$ for $1 \leq s \leq d$.
Since $H(\bi_\si)$ is commutative, at most one of $r-1$ and $r$ is an element of $\{r_1, \dots, r_a\}$.

    Suppose that $r \in \{r_1, \dots, r_a\}$. This is equivalent to $i_r = i_{r+1}$, $r \in \Leg(\si)$ and $r+1 \in \Arm(\si)$. Since the $\psi_{r_i}$ commute with each other, we may assume that $r = r_1$. We compute
    \begin{align*}
         y_r [\si] &= y_r \psi_r \psi_{r_2} \dots \psi_{r_a} [\si_\bi]\\
               &= (\psi_r y_{r+1} - 1) \psi_{r_2} \dots \psi_{r_a} [\si_\bi]\\
               &= \psi_r \psi_{r_2} \dots \psi_{r_a} y_{r+1} [\si_\bi] - \psi_{r_2} \dots \psi_{r_a} [\si_\bi]\\
               &= -\psi_{r_2} \dots \psi_{r_a} [\si_\bi] = -[s_r \si].
    \end{align*}

    Similarly, $r-1 \in \{r_1, \dots, r_a\}$ is equivalent to $i_{r-1} = i_r$, $r-1 \in \Leg(\si)$, and $r \in \Arm(\si)$. An argument similar to the above shows that $y_r [\si] = [s_{r-1} \si]$.
    If neither $r$ nor $r-1$ is among $\{r_1, \dots, r_a\}$, then $y_r$ commutes with each $\psi_{s_r}$, and thus $y_r \psi_{s_{r_1}} \dots \psi_{r_a} [\si_\bi] = \psi_{r_1} \dots \psi_{r_a} y_r [\tau] = 0$.

  (iii) Observe that $\wt([\si])$ begins with either $(0,1,\dots)$ or $(0,e-1,\dots)$. By weight, we must have $\psi_1 [\si] = 0$. We break down the rest of the proof into many cases.

  Case 1: Assume $r \in \Arm(\si)$ and $r+1 \in \Leg(\si)$. Clearly $\psi_r [\si] = [s_r \si]$.

  Case 2: Assume $r \in \Leg(\si)$, $r+1 \in \Arm(\si)$. Then $[\si] = \psi_r [s_r \si]$, and so $\psi_r [\si] = \psi_r^2 [s_r \si]$. If $i_r \nslash i_{r+1}$ then $\psi_r^2 [s_r \si] = [s_r \si]$, and if $i_r = i_{r+1}$ then $\psi_r^2 [s_r \si] = 0$.

    If $i_r \to i_{r+1}$, then $\psi_r^2 [s_r \si] = (y_r - y_{r+1})[s_r \si]$.
    Since $r \in \Arm(s_r \si)$, part (ii) of this theorem shows that if $y_r [s_r \si] \neq 0$ then $r-1 \in \Leg(s_r \si)$ and $i_{r-1} = i_{r+1}$. %Assume that $r-1 \in \Leg(s_r \si)$. We know that $r+1 \in \Leg(s_r \si)$ with weight $i_r$, and therefore $i_{r-1} = i_r + 1$. The condition $i_r \to i_{r+1}$ is equivalent to $i_{r+1} = i_r - 1$. Putting this all together, we have $i_{r-1} = i_{r+1} + 2$. Since we have assumed that $e \neq 2$, we cannot have $i_{r-1} = i_{r+1}$. Therefore $y_r [s_r \si] = 0$.
    This is seen to be impossible by considering weights, %(here it is important that $e \neq 2$), 
    and so $y_r [s_r \si] = 0$.
    Similarly, $r+1 \in \Leg(s_r \si)$, so part (ii) says that $y_{r+1}[s_r \si] = 0$ unless $r+2 \in \Arm(s_r \si)$ and $i_r = i_{r+2}$. %Assume that $r+2 \in \Arm(s_r \si)$, which is clearly equivalent to $r+2 \in \Arm(\si)$. We know that $r \in \Arm(s_r \si)$ with weight $i_{r+1}$, and therefore $i_{r+2} = i_{r+1} + 1$. The condition $i_r \to i_{r+1}$ is equivalent to $i_{r+1} = i_r - 1$, so we have $i_r = i_{r+2}$. Therefore $y_{r+1} [s_r \si] = -[s_{r+1} s_r \si]$ in this case.
    But $r+2 \in \Arm(s_r \si)$ if and only if $r+2 \in \Arm(\si)$, in which case $i_{r+2} = i_{r+1} + 1 = i_r$. We then conclude that $y_{r+1}[s_r \si] = -[s_{r+1} s_r \si]$. To summarize, if $i_r \to i_{r+1}$, then
      \[ \psi_r [\si] = \begin{cases}
             [s_{r+1} s_r \si],& \textup{ if $r+2 \in \Arm(\si)$}\\
         0,& \textup{ otherwise.}
      \end{cases} \]

    Now suppose $i_r \leftarrow i_{r+1}$. An argument similar to the one above shows that
      \[ \psi_r [\si] = \begin{cases}
             -[s_r s_{r+1} \si],& \textup{ if $r-1 \in \Leg(\si)$}\\
         0,& \textup{ otherwise.}
      \end{cases} \]

    Case 3: Assume $r, r+1 \in \Arm(\si)$. We prove the following statement by induction on $t$: If $\tau \in \Sh^\la$ satisfies $t, t+1 \in \Arm(\tau)$ we have 
    $$\psi_t [\tau] = \begin{cases} [s_t s_{t+1} \si], & \text{if $t-1 \in \Leg(\tau)$ and $i_{t-1} = i_t$}\\
      0, & \text{otherwise.}
    \end{cases}$$
    The base case of $t=1$ has been shown above.

    We proceed with the induction, fixing $t > 1$, $\tau \in \Sh^\la$ with $t, t+1 \in \Arm(\tau)$ and assuming the claim holds for all smaller values of $t$. Let $u$ be the largest value of $\Leg(\tau)$ which is less than $t$. If no such value exists, then $\psi_t [\tau] = \psi_t \psi_\tau z^\la = \psi_\tau \psi_t z^\la = 0$. Otherwise, there is a reduced expression for $\tau$ beginning with $s_u \dots s_{t-1} s_t$. Write $\tau' = s_t s_{t-1} \dots s_u \tau$. Then
    \begin{align*}
      \psi_t [\tau] &= \psi_t (\psi_u \dots \psi_{t-2}) \psi_{t-1} \psi_t [\tau'] = (\psi_u \dots \psi_{t-2}) \psi_t \psi_{t-1} \psi_t [\tau'] \\
        &= (\psi_u \dots \psi_{t-2})(\psi_{t-1} \psi_t \psi_{t-1} + \de_{i_t, i_u}) [\tau']\\
        &= (\psi_u \dots \psi_t) \psi_{t-1} [\tau'] + \de_{i_t, i_u} (\psi_u \dots \psi_{t-2}) [\tau'].
    \end{align*}

    Assume that $t-1 \in \Arm(\tau)$. Then the first term is zero by induction. Furthermore, the second term is zero unless $i_t = i_u$ and (by induction) $u=t-2$, $t-3 \in \Leg(\tau)$, and $i_{t-3} = i_{t-1}$. But since $i_{t-1} = i_t - 1$ and $i_{t-3} = i_u + 1$, this is impossible. Therefore the second term is zero as well.

    Now suppose that $u = t-1$. We can see that the first term is zero using an argument similar to the above. If the second term appears, it is visibly nonzero. We have thus proved the induction step.
    
    Applying this to case 3, we have
    $$\psi_r [\si] = \begin{cases}[s_r s_{r-1} \si],& \textup{ if $r-1 \in \Leg(\si)$, and $i_r = i_{r-1}$}\\
     0,& \text{otherwise.} \end{cases}$$

    Case 4: Assume $r, r+1 \in \Leg(\si)$. This situation is entirely analogous to case 3, except now induction runs backwards from $t = d-1$ to $t = 2$. We find that
    $$\psi_r [\si] = \begin{cases}-[s_{r-1} s_r \si],& \textup{ if $r+2 \in \Arm(\si)$, and $i_{r+1} = i_{r+2}$}\\
     0,& \text{otherwise.} \end{cases}$$

\end{proof}

\begin{Corollary}\label{cor:Extreme}
  Suppose that $e(\bi)S^\la \neq 0$. Then
  \[
    \{v \in e(\bi)S^\la\ |\ y_r v = 0\ \textup{ for } r = 1, \dots, d\} = \O [\si_\bi].
  \]
\end{Corollary}

\begin{proof}
  By Theorem~\ref{thm:KLRonBasis}, we have $y_r [\si_\bi] = 0$ for $r = 1, \dots, d$.
  Suppose given
  \[
    v = \sum_{\si \in \Sh^\la}{c_\si [\si]} \in e(\bi)S^\la,
  \]
  and let $h = s_{r_1} \dots s_{r_m} \in H(\bi)$ be a reduced expression for an element of maximal length for which $c_{h\si_\bi} \neq 0$. Apply Theorem~\ref{thm:KLRonBasis} to obtain $y_{r_1} \dots y_{r_m} v = (-1)^m c_{h\si_\bi} [\si_\bi] \neq 0$. Therefore $v$ is killed by every $y_r$ if and only if $v \in \O[\si_\bi]$.
\end{proof}

\section{Homomorphisms into $S^\la$}

We again fix two integers $d > k \geq 0$,  $\la := (d-k,1^k)$ and $\al := \cont(\la)$. We also fix $\mu = (\mu_1, \dots, \mu_N)$ a partition of $d$. The goal of this section is to determine $\Hom_{R_\al}(S^\mu, S^\la)$, and the answer is given in Theorem~\ref{thm:mainHom}. In order for a nonzero homomorphism to exist $\bi^\mu$ must be a weight of $S^\la$. We assume this, so in particular we are able to define $\si_\mu := \si_{\bi^\mu}$. This determines a mapping from $\mu$ to $\la$ in which the node containing $r$ in $\T^\mu$ is sent to the node containing $r$ in $\si_\mu \T^\la$. We define \emph{arm nodes} and \emph{leg nodes} of $\mu$ to be the nodes which under this map are sent to the arm or the leg of $\la$, respectively.

\begin{Theorem}\label{thm:mainHom}
  There is a homomorphism $S^\mu \to S^\la$ satisfying $z^\mu \mapsto v$ for $v \neq 0$ if and only if
  \begin{enumerate}
    \item there exist $c \in \{1, \dots, k+1\}$, $\ba \in (\Z_{>0})^c$, and $0 \leq m < e$ such that $v \in \Ann_\O(\Gc(\ba)) [\si_\mu]$ and
$$\mu = (a_1 e, \dots, a_{c-1} e, a_c e - m, 1^{k-c+1}),$$
    \item $e$ divides $d$, there exist $c \in \{1, \dots, k\}$, $\ba \in (\Z_{>0})^c$, and $0 \leq m < e$ such that $v \in \Ann_\O(\Gc(\ba)) [\si_\mu]$ and
$$\mu = (a_1 e, \dots, a_{c-1} e, a_c e - m, 1^{k-c+2}),\ \text{or}$$
    \item $N > k+1$, there exist $\ba \in (\Z_{>0})^N$ and $0 \leq m < e$ such that $v \in \Ann_\O(\Gc(\ba)) [\si_\mu]$ and
$$\mu = (a_1 e, \dots, a_k e, a_{k+1}e - 1, \dots, a_{N-1} e - 1, a_N e - 1 - m).$$
  \end{enumerate}
\end{Theorem}

The proof of this will follow from several lemmas.
Recall the ideals $J_i^\mu$ of Definition~\ref{DSpecht}.

\begin{Lemma}\label{lem:rel123}
  A nonzero $v \in S^\la$ satisfies $(J_1^\mu + J_2^\mu + J_3^\mu)v = 0$ if and only if
$v \in \O [\si_\mu]$, and
all leg nodes of $\mu$ are in the first column of its Young diagram.
\end{Lemma}

\begin{proof}
Assume that $(J_1^\mu + J_2^\mu + J_3^\mu)v = 0$. 
Corollary~\ref{cor:Extreme} implies that $v=\ga [\si_\mu]$ for some $\ga \in \O$. 
By considering weights we see that $(1,2)$ cannot be a leg node. Suppose that $(a, b) \neq (1,2)$ is a leg node of $\mu$ with $b \geq 2$, and define $r = \T^\mu(a,b)$. Then since $J_3^\mu v = 0$ we must have $\psi_{r-1} v = 0$.
Furthermore $(a, b-1)$ must be an arm node, for in $\bi^\mu=(i_1, \dots, i_d)$ we have $i_r = i_{r-1} + 1$, whereas if they were both leg nodes we would have $i_r = i_{r-1} - 1$. By Theorem~\ref{thm:KLRonBasis} we have $\psi_{r-1} \ga [\si_{\mu}] = \ga [s_{r-1} \si_{\mu}] \neq 0$, a contradiction.

Conversely, assume that $v=\ga [\si_\mu]$ and all leg nodes of $\mu$ are in the first column. Suppose that $r$ and $r+1$ are in the same row of $\T^\mu$, and so in particular $r+1$ is an arm node. If the node of $\mu$ containing $r$ in $\T^\mu$ is a leg node %and $r\rightarrow_{\T^\mu} r+1$
then by considering Theorem~\ref{thm:KLRonBasis} we see that $\psi_r v = 0$ unless the node containing $r-1$ is also a leg node. But then the nodes containing $r-1$ and $r$ are both in the first column of $\T^\mu$ forcing $r+1$ (and every subsequent node) to also be in the first column. This contradicts the assumption that $r$ and $r+1$ appear in the same row of $T^\mu$, and so necessarily $\psi_r v = 0$.
  
  If the node of $\mu$ containing $r$ in $\T^\mu$ is an arm node, then write $\bi^\mu = (i_1, \dots, i_d)$. Referring again to Theorem~\ref{thm:KLRonBasis} we see that $\psi_r v = 0$ unless the node containing $r-1$ is a leg node, and furthermore $i_{r-1} = i_r$. But if the the node containing $r-1$ is a leg node, then by assumption it is in the first column. Once again the assumption that $r$ and $r+1$ are in the same row of $T^\mu$ forces the node containing $r-1$ to be directly to the left of the node containing $r$. But in this case $i_{r-1} = i_r - 1$ and once again we see that $\psi_r v = 0$.
\end{proof}

\begin{Proposition}\label{prop:legNodes}
  If there is a nonzero element $v \in S^\la$ with $J^\mu v = 0$, then the leg nodes of $\mu$ are precisely $(2,1), (3,1), \dots, (k + 1,1)$.
\end{Proposition}

\begin{proof}
  Let $v = \ga [\si_{\mu}]$ be such an element. By Lemma~\ref{lem:rel123}, the leg nodes appear in the first column of $\mu$. Suppose that $(a, 1)$ is a leg node. Then $a \neq 1$, and so we may consider the Garnir node $A = (a-1, 1)$. Define $r = \T^\mu(a-1,1)$ and $s = \T^\mu(a,1)$. It is clear that $g^A = \psi^{T^A} = \psi_{(r, r+1, \dots, s)}$ where $(r, r+1, \dots, s)$ is a cycle of the entire Garnir belt. In particular, if $a \geq 3$ and $(a-1,1)$ is not a leg node, then $g^A [\si_{\mu}] = [(r, r+1, \dots, s)\si_{\mu}]$.
But this is a nonzero basis element by Lemma~\ref{lem:rel123}, contradicting $J^\mu_4 v = 0$. This contradiction shows that for every leg node $(a,1)$ with $a \geq 3$, the node $(a-1,1)$ is also a leg node, which is equivalent to the statement of the proposition.
\end{proof}

%One can show that $\mu = (\mu_1, \dots, \mu_N) \ledom \la$ if and only if $N \geq k+1$. Therefore, Proposition~\ref{prop:legNodes} implies that $\Hom(S^\mu, S^\la) \neq 0$ only if $\mu \ledom \la$. This is similar to a general result for cellular algebras, see \cite[Proposition 2.6 (i)]{GL}.

As a consequence of this proposition, we see that the leg nodes of $\mu$ are mapped to exactly the same nodes of $\la$, and therefore that $\si_{\mu} \T^\la$ is the tableau obtained by moving all of the arm nodes of $\T^\mu$ to the first row.
\iffalse{This is illustrated in the following example, where $\mu = (7, 5, 3, 3, 1)$ and $\la = (15, 1^{4})$. Any homomorphism $S^\mu \to S^\la$ must send $z^{(7,5,3,3,1)}$ to a multiple of the $\la$-tableau as depicted below.
\[
\begin{tikzpicture}[scale=0.5,draw/.append style={thick,black},baseline=1mm]
  \fill[red!50](0.5,-0.5)rectangle(1.5,-4.5);
  \newcount\col
  \foreach\Row/\row in {{1,...,7}/0,{8,...,12}/-1,{13,14,15}/-2,{16,17,18}/-3,{19}/-4} {
     \col=1
     \foreach\k in \Row {
        \draw(\the\col,\row)+(-.5,-.5)rectangle++(.5,.5);
        \draw(\the\col,\row)node{\k};
        \global\advance\col by 1
      }
   }
  \draw[|->](7,-2)--(10,-2)node[right]{$\ga$};
  \newcount\col
  \foreach\Row/\row in {{1,...,7,9,10,11,12,14,15,17,18}/0,{8}/-1,{13}/-2,{16}/-3,{19}/-4} {
     \col=1
     \foreach\k in \Row {
        \draw(11+\the\col,\row)+(-.5,-.5)rectangle++(.5,.5);
        \draw(11+\the\col,\row)node{\k};
        \global\advance\col by 1
      }
   }
\end{tikzpicture}
\]
}\fi
This suggests the following definition.

\begin{Definition}
  Let $\mu = (\mu_1, \dots, \mu_N)$ be a partition of $d$ with $N \geq k+1$. Define $\T^\la_\mu$ to be the $\la$-tableau obtained as follows. Let $S = \{\T^\mu(A)\ |\ A \in \mu \setminus \la\}$. We fill in the rows of $\la$ from left to right and from top to bottom using the following procedure. Let $A \in \la$ and suppose we have filled in all of the nodes to the left of $A$ or in a previous row. If $A \in \mu$ then define $\T^\la_\mu(A) = \T^\mu(A)$. Otherwise define $\T^\la_\mu(A)$ to be the smallest element of $S$, and then delete this value from $S$. Repeat this process until all nodes have been filled. %Properties of the dominance order ensure that $\T^\la_\mu$ is standard.
  Furthermore, define $\si^\la_\mu := w^{\T^\la_\mu} \in \Si_n$.
\end{Definition}

%\begin{Corollary}
%  Given any partition $\mu$ of $n$, if there is a homomorphism $S^\mu \to S^\la$, then it is non-negatively graded.
%\end{Corollary}
%\begin{proof}
%  We prove this by comparing $\deg(\T^\la)$ with $\deg(\T^\la_\mu)$, using induction on the number of rows in $\mu$. ETC.
%\end{proof}

If we have a nonzero homomorphism $S^\mu \to S^\la$, then Corollary~\ref{cor:Extreme} tells us it must send $z^\mu$ to a multiple of $[\si_{\mu}]$. On the other hand, Proposition~\ref{prop:legNodes} requires the image of $z^\mu$ to be a multiple of $[\si^\la_\mu]$. Since $[\si_{\mu}]$ and $[\si^\la_\mu]$ are elements of an $\O$-basis of $S^\la$, a necessary condition for a homomorphism is that $[\si_{\mu}] = [\si^\la_\mu]$. The next two results tell us when this happens.

\begin{Lemma}\label{lem:muForms}
  Let $\mu = (\mu_1, \dots, \mu_N)$ be a partition of $d$ such that $N \geq k+1$. Then $[\si^\la_\mu]$ has weight $\bi^\mu$ if and only if $\mu$ is of one of the following three forms.% (the nodes $(2,1), \dots, (k+1,1)$ are colored red in each of the diagrams).
  \begin{enumerate}
    \item $N = k+1$, and there exist $c \in \{1, \dots, k+1\}$, $a_1, \dots, a_c \in \Z_{>0}$ and $0 \leq m < e$ such that $\mu = (a_1 e, \dots, a_{c-1} e, a_c e - m, 1^{k-c+1})$. 
\iffalse
\[
\begin{tikzpicture}[scale=0.5,draw/.append style={thick,black},baseline=1mm]
  \fill[red!50](0.5,-0.5)rectangle(1.5,-8.5);
  \newcount\col
  \foreach\Row/\row in {{1,...,16}/0,{1,...,12}/-1,{1,...,12}/-2,{1,...,8}/-3,{1,2,3}/-4,{1}/-5,{1}/-6,{1}/-7,{1}/-8} {
     \col=1
     \foreach\k in \Row {
        \draw(\the\col,\row)+(-.5,-.5)rectangle++(.5,.5);
%        \draw(\the\col,\row)node{\k};
        \global\advance\col by 1
      }
   }
   \draw(0,0)node[left]{$a_1 e$};
   \draw(0,-1)node[left]{$a_2 e$};
   \draw(0,-2)node[left]{$\dots$};
   \draw(0,-3)node[left]{$a_{c-1} e$};
   \draw(0,-4)node[left]{$a_c e - m$};
   \draw(0,-5)node[left]{$1$};
   \draw(0,-6)node[left]{$1$};
   \draw(0,-7)node[left]{$\dots$};
   \draw(0,-8)node[left]{$1$};
\end{tikzpicture}
\]
\fi

    \item $N > k+1$ and $\mu_{k+1}=1$, and there exist $c \in \{1, \dots, k\}$, $a_1, \dots, a_c \in \Z_{>0}$ and $0 \leq m < e$ satisfying $m+c \equiv k+2$ (mod $e$) such that $\mu = (a_1 e, \dots, a_{c-1} e, a_c e - m, 1^{k-c+1}, 1)$.
\iffalse
\[
\begin{tikzpicture}[scale=0.5,draw/.append style={thick,black},baseline=1mm]
  \fill[red!50](0.5,-0.5)rectangle(1.5,-8.5);
  \newcount\col
  \foreach\Row/\row in {{1,...,16}/0,{1,...,12}/-1,{1,...,12}/-2,{1,...,8}/-3,{1,2,3}/-4,{1}/-5,{1}/-6,{1}/-7,{1}/-8,{1}/-9} {
     \col=1
     \foreach\k in \Row {
        \draw(\the\col,\row)+(-.5,-.5)rectangle++(.5,.5);
%        \draw(\the\col,\row)node{\k};
        \global\advance\col by 1
      }
   }
   \draw(0,0)node[left]{$a_1 e$};
   \draw(0,-1)node[left]{$a_2 e$};
   \draw(0,-2)node[left]{$\dots$};
   \draw(0,-3)node[left]{$a_{c-1} e$};
   \draw(0,-4)node[left]{$a_c e - m$};
   \draw(0,-5)node[left]{$1$};
   \draw(0,-6)node[left]{$1$};
   \draw(0,-7)node[left]{$\dots$};
   \draw(0,-8)node[left]{$1$};
   \draw(0,-9)node[left]{$1$};
\end{tikzpicture}
\]
\fi
    \item $N > k+1$ and $\mu_{k+1} > 1$, and there exist $a_1, \dots, a_N \in \Z_{>0}$ and $0 \leq m < e$ such that  $\mu = (a_1 e, \dots, a_k e, a_{k+1}e - 1, \dots, a_{N-1} e - 1, a_N e - 1 - m)$.
\iffalse
\[
\begin{tikzpicture}[scale=0.5,draw/.append style={thick,black},baseline=1mm]
  \fill[red!50](0.5,-0.5)rectangle(1.5,-8.5);
  \newcount\col
  \foreach\Row/\row in {{1,...,16}/0,{1,...,16}/-1,{1,...,16}/-2,{1,...,12}/-3,{1,...,12}/-4,{1,...,12}/-5,{1,...,8}/-6,{1,...,8}/-7,{1,...,7}/-8,{1,...,7}/-9,{1,...,7}/-10,{1,2,3}/-11,{1,2,3}/-12,{1,2}/-13} {
     \col=1
     \foreach\k in \Row {
        \draw(\the\col,\row)+(-.5,-.5)rectangle++(.5,.5);
%        \draw(\the\col,\row)node{\k};
        \global\advance\col by 1
      }
   }
   \draw(0,0)node[left]{$a_1 e$};
   \draw(0,-1)node[left]{$a_2 e$};
   \draw(0,-2)node[left]{$a_3 e$};
   \draw(0,-3)node[left]{$\dots$};
   \draw(0,-4)node[left]{$\dots$};
   \draw(0,-5)node[left]{$\dots$};
   \draw(0,-6)node[left]{$a_{k-1} e$};
   \draw(0,-7)node[left]{$a_k e$};
   \draw(0,-8)node[left]{$a_{k+1} e - 1$};
   \draw(0,-9)node[left]{$a_{k+2} e - 1$};
   \draw(0,-10)node[left]{$\dots$};
   \draw(0,-11)node[left]{$\dots$};
   \draw(0,-12)node[left]{$a_{N-1} e - 1$};
   \draw(0,-13)node[left]{$a_N e - m$};
\end{tikzpicture}
\]
\fi
  \end{enumerate}
\end{Lemma}

\begin{proof}

  Define $(i_1, \dots, i_d) := \wt([\si^\la_\mu])$ and $(j_1, \dots, j_d) := \bi^\mu$. Define $r_x = \T^\mu(x+1,1)$ for $x = 1, 2, \dots, k$. Since $r_x = \T^\mu(x+1,1) = \T^\la_\mu(x+1,1)$, it is immediate that $i_{r_x} = j_{r_x}$ for $x = 1, \dots, k$. Letting $S = \{1, 2, \dots, d\} \setminus \{r_1, \dots, r_k\}$, we therefore have that $\wt([\si^\la_\mu]) = \bi^\mu$ if and only if $i_s = j_s$ for all $s \in S$.
  
  For $p = 1, 2, \dots, n-k$, let $s_p$ be the $p^{th}$ element of the set $S$ (under the usual ordering). It is clear that $s_p = \T^\la_\mu(1,p)$. In particular, we see that $i_{s_{p+1}} = i_{s_p} + 1$ for $p = 1, \dots, d-k-1$. Since $j_{s_1} = i_{s_1}$, we have $\wt([\si^\la_\mu]) = \bi^\mu$ if and only if $j_{s_{p+1}} = j_{s_p} + 1$ for $p = 1, \dots, d-k-1$.

  Set $r_0 = 0$. Note that the condition $r_1 \neq r_0 + 1$ is automatically satisfied, so $x_0 = \max \{x\ |\ r_x \neq r_{x-1} + 1 \textup{ where } 1 \leq x \leq k\}$ is well defined. We divide $S$ as follows
\begin{align*}
  (s_1, s_2, \dots, s_{n-k}) = (1, & 2, \dots, r_1-1, \\
                                   & r_1+1, \dots, r_2-1, \\
                                   & \dots, \\
                                   & r_{x_0-1}+1, \dots, r_{x_0}-1, \\
                             & r_k+1, \dots, d).
                             %t_{k+1}, \\
                             %& t_{k+1} + 1, \dots, t_{k+2}, \\
                             %& \dots, \\
                             %& t_{N-1}+1, \dots, t_N = n).
\end{align*}
Now, $(r_{x-1}+1, \dots, r_x-1) = (\T^\mu(x, 2), \dots, \T^\mu(x, \mu_x))$ for $x \leq x_0$, so it is clear that $j_{p+1} = j_p + 1$ when $r_{x-1}+1 \leq p \leq r_x-2$. Furthermore, for $x < x_0$ we have $r_x - 1 = \T^\mu(x, \mu_x)$ while $r_x + 1 = \T^\mu(x+1,2)$. If we let $i = \res(x,1)$, then $j_{r_x-1} = \res(x, \mu_x) \equiv i + \mu_x - 1$ (mod $e$), and $j_{r_x+1} = \res(x+1,2) = i$. Thus $j_{r_x+1} = j_{r_x-1} + 1$ if and only if $\mu_x = a_x e$ for some $a_x \in \Z_{>0}$.

  At this point we have shown that $i_p = j_p$ for $p = 1, \dots, r_k$ if and only if for every $x < x_0$ there exist $a_x \in \Z_{>0}$ such that $\mu_x = a_x e$.

  Suppose that $N = k+1$. If $r_k = d$, then we set $c = x_0$ choose $a_c, m$ so that $\mu_c = a_c e - m$. This clearly is of the form in part (i) of the Lemma. Otherwise, $x_0 = k$ and $r_k + 1, \dots, d$ are all on row $k+1$ of $\T^\mu$. The same analysis as above shows that $[\si^\la_\mu]$ has weight $\bi^\mu$ if and only if for every $x \leq k$ there exist $a_x \in \Z_{>0}$ such that $\mu_x = a_x e$. In this case we set $c = k+1$ and choose $a_c, m$ as required. We have established part (i) of the Lemma.

  Suppose that $N > k+1$ and $\mu_{k+1} = 1$. Set $c = x_0$ choose $a_c, m$ so that $\mu_c = a_c e - m$. Then $(r_k+1, \dots, d) = (\T^\mu(k+2,1), \dots, \T^\mu(N,1))$. Let $i = \res(c,1)$. Then $j_{r_c-1} = \res(c, \mu_c) \equiv i + \mu_c - 1$ (mod $e$), and $j_{r_k+1} \equiv i + k + 1 - c$ (mod $e$). Thus $j_{r_k+1} = j_{r_{x_0}-1} + 1$ if and only if $\mu_c \equiv k + 1 - c$ (mod $e$). But $\mu_c = a_c e - m$, so this is equivalent to $c + m \equiv k + 1$ (mod $e$). This is exactly the requirement in part (ii) of the Lemma. Finally, suppose that $N \geq k+3$. Then $r_k+1 = \T^\mu(k+2,1)$ and $r_k+2 = \T^\mu(k+3,1)$. In this case we see that $j_{r_k + 2} = j_{r_k + 1} - 1 \neq j_{r_k + 1} + 1$. Thus part (ii) of the Lemma is finished.

  Finally, suppose that $N > k+1$ and $\mu_{k+1} > 1$. Then $x_0 = k$. Once again the analysis of the $N = k+1$ case shows that $i_p = j_p$ for $p = 1, \dots, r_k+1$ if and only if for every $x \leq k$ there exist $a_x \in \Z_{>0}$ such that $\mu_x = a_x e$. Furthermore, a similar analysis shows that $i_p = j_p$ for $p = 1, \dots, d$ if and only if additionally for every $k+1 \leq x < N$ there exist $a_x \in \Z_{>0}$ such that $\mu_x = a_x e - 1$. This is exactly the situation of part (iii).
\end{proof}

\begin{Corollary}\label{cor:muForms}
  If $\mu$ is as in Lemma~\ref{lem:muForms}, then $(J_2^\mu + J_3^\mu)[\si^\la_\mu] = 0$. In particular, $[\si^\la_\mu] = [\si_{\mu}]$.
\end{Corollary}

\begin{proof}
  Let $(i_1, \dots, i_d) = \wt([\si^\la_\mu])$. It is clear that in each of the cases above if $i_r = i_{r+1}$ then $r \in \Arm(\si^\la_\mu)$ and $r+1 \in \Leg(\si^\la_\mu)$. This implies that $J^\mu_2[\si^\la_\mu] = 0$. In turn, this implies that $[\si^\la_\mu] = [\si_{\mu}]$, by Corollary~\ref{cor:Extreme}. Lemma~\ref{lem:rel123} now shows us that $J^\mu_3[\si^\la_\mu] = 0$.
\end{proof}

As expected, the Garnir relations are the most difficult to verify.

\begin{Lemma}\label{lem:GarnirPsi}
  Suppose that $\mu$ has one of the forms in Lemma~\ref{lem:muForms} and let $A = (x,y) \in \mu$ be a Garnir node. Define $r, s,$ and $t$ to be the values in $\T^\mu$ of the nodes $(x,y), (x+1,1),$ and $(x+1,y)$ respectively so that the Garnir belt is as in the following picture
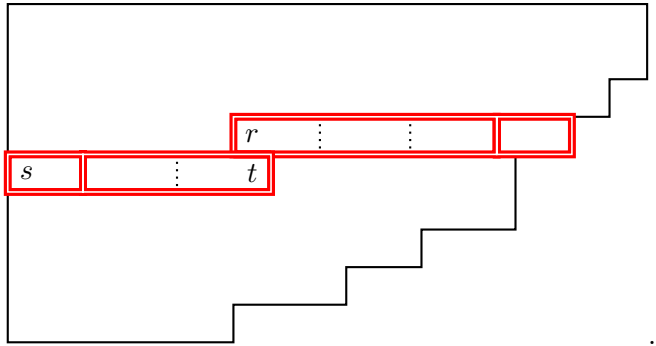
\[
  \begin{tikzpicture}[scale=0.5,draw/.append style={thick,black}]
    \draw (0.5,-0.5)--++(17,0)--++(0,-2)--++(-1,0)--++(0,-1)--++(-1,0)--++(0,-1)--++(-1.5,0)
            --++(0,-2)--++(-2.5,0)--++(0,-1)--++(-2,0)--++(0,-1)--++(-3,0)--++(0,-1)--++(-6,-0)--cycle;
    \draw[red,double,very thick] (6.5,-4.5)--++(1,0);
    \draw[red,double,very thick] (13.5,-3.5)--++(0,-1);
    \draw[dotted] (8.8,-3.5)--++(0,-1);
    \draw[dotted] (11.2,-3.5)--++(0,-1);
    \draw[red,double,very thick] (2.5,-4.5)--++(0,-1);
    \draw[dotted] (5,-4.5)--++(0,-1);
%    \draw[red,double,very thick]
%       (6.5,-3.5)--++(9,0)--++(0,-1)--++(-8,0)--++(0,-1)--++(-7,0)--++(0,1)--++(6,0)--cycle;
    \draw[red,double,very thick]
       (6.5,-3.5)rectangle++(7,-1)rectangle++(2,1);
    \draw[red,double,very thick]
       (0.5,-4.5)rectangle++(2,-1)rectangle++(5,1);
    \draw (7,-4) node {$r$};
    \draw (1,-5) node {$s$};
    \draw (7,-5) node {$t$};
  \end{tikzpicture}.
\]
  Then
  \[ 
    \psi^{\T^A} [\si^\la_\mu] = \begin{cases}
      [(t+1, t, \dots, s) \si^\la_\mu],& \textup{if $x \leq k$, $y \equiv 0$ (mod $e$), and $y < \mu_{x+1}$} \\
      [(r, r+1, \dots, s) \si^\la_\mu],& \textup{if $x \leq k$, $y \equiv 1$ (mod $e$), and $y > 1$} \\
      [\si^\la_\mu],& \textup{if $x > k$ and $y \equiv 0$ (mod $e$)} \\
      0,& \textup{otherwise} \\
    \end{cases}
  \]
\end{Lemma}

\begin{proof}
  We begin with the case $x \leq k$ and $y \equiv 0$ (mod $e$). Define $s_p = s + pe$ for $p \in \N$. Using braid diagrams, we compute $\psi^{\T^A}[\si^\la_\mu]$.%We use Theorem~\ref{thm:KLRonBasis}(iii) extensively to compute $\psi^{\T^A}[\si^\la_\mu]$ below. %Section 2 is empty
\begin{equation}\label{eq:Gar_st}
  \begin{braid}\tikzset{scale=0.6,baseline=12mm}
    \draw(0,8)node[left,font=\normalsize]{$[\si^\la_\mu]$};
    \draw(0,2)node[left,font=\normalsize]{$\psi^{\T^A}$};
    \draw(1,10)--(5,6); \draw(5,4)--(5,0);
    \draw(2,10)node[above]{$i$}--(6,6); \draw(6,5)node[color=red]{$s_0$}; \draw(6,4)--(13,0);
    \draw(15,10)node[above]{$i$}--(7,6); \draw(7,4)--(6,0);
    \draw(3,10)--(8,6); \draw(8,4)--(7,0);
    \draw(4,10)--(9,6); \draw(9,4)--(8,0);
    \draw(7,10)--(12,6); \draw(12,4)--(11,0);
    \draw(8,10)--(13,6); \draw(13,5)node[color=red]{$t$}; \draw(13,4)--(12,0);
    \draw[dotted](0,10)--(0.5,10);
    \draw[dotted](0,6)--(4.5,6);
    \draw[dotted](0,4)--(4.5,4);
    \draw[dotted](0,0)--(4.5,0);
    \draw[dotted](4.5,10)--(6.5,10);
    \draw[dotted](9.5,6)--(11.5,6);
    \draw[dotted](9.5,4)--(11.5,4);
    \draw[dotted](8.5,0)--(10.5,0);
    \draw[dotted](8.5,10)--(14.5,10);
    \draw[dotted](15.5,10)--(16,10);
    \draw[dotted](13.5,6)--(16,6);
    \draw[dotted](13.5,4)--(16,4);
    \draw[dotted](13.5,0)--(16,0);
  \end{braid}
  \begin{braid}\tikzset{scale=0.6,baseline=12mm}
    \draw(0,5)node[font=\normalsize]{$=$};
  \end{braid}
  \begin{braid}\tikzset{scale=0.6,baseline=12mm}
    \draw[dotted](0,10)--(0.5,10);
    \draw[dotted](0,5)--(4.5,5);
    \draw[dotted](0,0)--(4.5,0);

    \draw(1,10)--(5,5)--(5,0);
    \draw(2,10)node[above]{$i$}--(7,5)--(13,0)node[below,color=red]{$t$};
    \draw(3,10)--(8,5)--(7,0);
    \draw(4,10)--(9,5)--(8,0);

    \draw[dotted](4.5,10)--(6.5,10);
    \draw[dotted](9.5,5)--(11.5,5);
    \draw[dotted](8.5,0)--(10.5,0);

    \draw(7,10)--(12,5)--(11,0);
    \draw(8,10)--(13,5)--(12,0);

    \draw[dotted](8.5,10)--(14.5,10);

    \draw(15,10)node[above]{$i$}--(6,5)--(6,0)node[below,color=red]{$s_0$};

    \draw[dotted](15.5,10)--(16,10);
    \draw[dotted](13.5,5)--(16,5);
    \draw[dotted](13.5,0)--(16,0);
  \end{braid}
\end{equation}
\[
  \begin{braid}\tikzset{scale=0.6,baseline=12mm}
    \draw(0,4)node[font=\normalsize]{$=$};
  \end{braid}
  \begin{braid}\tikzset{scale=0.6,baseline=12mm}
    \draw(1,8)--(5,4)--(5,0);
    \draw(2,8)node[above]{$i$}--(6,4)--(6,0)node[below,color=red]{$s_0$};
    \draw(3,8)--(7,4)--(7,0);
    \draw(15,8)node[above]{$i$}--(8,4)--(13,0)node[below,color=red]{$t$};
    \draw(4,8)--(9,4)--(8,0);
    \draw(7,8)--(12,4)--(11,0);
    \draw(8,8)--(13,4)--(12,0);
    \draw[dotted](0,8)--(0.5,8);
    \draw[dotted](0,4)--(4.5,4);
    \draw[dotted](0,0)--(4.5,0);
    \draw[dotted](4.5,8)--(6.5,8);
    \draw[dotted](9.5,4)--(11.5,4);
    \draw[dotted](8.5,0)--(10.5,0);
    \draw[dotted](8.5,8)--(14.5,8);
    \draw[dotted](15.5,8)--(16,8);
    \draw[dotted](13.5,4)--(16,4);
    \draw[dotted](13.5,0)--(16,0);
  \end{braid}
  \begin{braid}\tikzset{scale=0.6,baseline=12mm}
    \draw(0,4)node[font=\normalsize]{$+$};
  \end{braid}
  \begin{braid}\tikzset{scale=0.6,baseline=12mm}
    \draw(1,8)--(5,4)--(5,0);
    \draw(2,8)node[above]{$i$}--(6,4)--(13,0)node[below,color=red]{$t$};
    \draw(3,8)--(7,4)--(6,3)--(7,0);
    \draw(15,8)node[above]{$i$}--(8,4)--(6,0)node[below,color=red]{$s_0$};
    \draw(4,8)--(9,4)--(8,0);
    \draw(7,8)--(12,4)--(11,0);
    \draw(8,8)--(13,4)--(12,0);
    \draw[dotted](0,8)--(0.5,8);
    \draw[dotted](0,4)--(4.5,4);
    \draw[dotted](0,0)--(4.5,0);
    \draw[dotted](4.5,8)--(6.5,8);
    \draw[dotted](9.5,4)--(11.5,4);
    \draw[dotted](8.5,0)--(10.5,0);
    \draw[dotted](8.5,8)--(14.5,8);
    \draw[dotted](15.5,8)--(16,8);
    \draw[dotted](13.5,4)--(16,4);
    \draw[dotted](13.5,0)--(16,0);
  \end{braid}.
\]
The second term is zero by Theorem~\ref{thm:KLRonBasis}(iii), and the first is equal to
\[
  \begin{braid}\tikzset{scale=0.6,baseline=12mm}
    \draw[dotted](0,8)--(0.5,8);
    \draw[dotted](0,4)--(4.5,4);
    \draw[dotted](0,0)--(4.5,0);
    \draw(1,8)--(5,4)--(5,0);
    \draw(2,8)node[above]{$i$}--(6,4)--(6,0)node[below,color=red]{$s_0$};
    \draw(3,8)--(7,4)--(7,0);
    \draw[dotted](3.5,8)--(5.5,8);
    \draw[dotted](7.5,4)--(9.5,4);
    \draw[dotted](7.5,0)--(9.5,0);
    \draw(6,8)--(10,4)--(10,0);
    \draw[dotted](12.5,8)--(18.5,8);
    \draw(19,8)node[above]{$i$}--(11,4)--(17,0)node[below,color=red]{$t$};
    \draw(7,8)--(12,4)--(11,0);
    \draw(8,8)node[above]{$i$}--(13,4)--(12,0);
    \draw(13,0)node[below,color=red]{$s_1$};
    \draw[dotted](8.5,8)--(10.5,8);
    \draw[dotted](13.5,4)--(15.5,4);
    \draw[dotted](12.5,0)--(14.5,0);
    \draw(11,8)--(16,4)--(15,0);
    \draw(12,8)--(17,4)--(16,0);
    \draw[dotted](19.5,8)--(20,8);
    \draw[dotted](17.5,4)--(20,4);
    \draw[dotted](17.5,0)--(20,0);
  \end{braid}
  \begin{braid}\tikzset{scale=0.6,baseline=12mm}
    \draw(0,4)node[font=\normalsize]{$=$};
  \end{braid}
  \begin{braid}\tikzset{scale=0.6,baseline=12mm}
    \draw[dotted](0,8)--(0.5,8);
    \draw[dotted](0,4)--(4.5,4);
    \draw[dotted](0,0)--(4.5,0);
    \draw(1,8)--(5,4)--(5,0);
    \draw(2,8)node[above]{$i$}--(6,4)--(6,0)node[below,color=red]{$s_0$};
    \draw(3,8)--(7,4)--(7,0);
    \draw[dotted](3.5,8)--(5.5,8);
    \draw[dotted](7.5,4)--(9.5,4);
    \draw[dotted](7.5,0)--(9.5,0);
    \draw(6,8)--(10,4)--(10,0);
    \draw(7,8)--(11,4)--(11,0);
    \draw(8,8)node[above]{$i$}--(12,4)--(18,0)node[below,color=red]{$t$};
    \draw[dotted](13.5,8)--(19.5,8);
    \draw(20,8)node[above]{$i$}--(13,4)--(12,0)node[below,color=red]{$s_1$};
    \draw(9,8)--(14,4)--(13,0);
    \draw[dotted](9.5,8)--(11.5,8);
    \draw[dotted](14.5,4)--(16.5,4);
    \draw[dotted](13.5,0)--(15.5,0);
    \draw(12,8)--(17,4)--(16,0);
    \draw(13,8)--(18,4)--(17,0);
    \draw[dotted](20.5,8)--(21,8);
    \draw[dotted](18.5,4)--(21,4);
    \draw[dotted](18.5,0)--(21,0);
  \end{braid},
\]
  using the defining relations of $R_\al$ and Theorem~\ref{thm:KLRonBasis}(ii). We note that this is of the same form as \eqref{eq:Gar_st} above. An easy induction argument allows us to conclude that this is equal to
\[
  \begin{braid}\tikzset{scale=0.6,baseline=12mm}
    \draw[dotted](0,8)--(0.5,8);
    \draw[dotted](0,4)--(4.5,4);
    \draw[dotted](0,0)--(4.5,0);
    \draw(1,8)--(5,4)--(5,0);
    \draw(2,8)node[above]{$i$}--(6,4)--(6,0)node[below,color=red]{$s_0$};
    \draw(3,8)--(7,4)--(7,0);
    \draw[dotted](3.5,8)--(5.5,8);
    \draw[dotted](7.5,4)--(9.5,4);
    \draw[dotted](7.5,0)--(9.5,0);
    \draw(6,8)--(10,4)--(10,0);
    \draw(7,8)node[above]{$i-1$}--(12,4)--(11,0);
    \draw[dotted](7.5,8)--(13.5,8);
    \draw(14,8)node[above]{$i$}--(11,4)--(12,0)node[below,color=red]{$t$};
    \draw[dotted](14.5,8)--(15,8);
    \draw[dotted](12.5,4)--(15,4);
    \draw[dotted](12.5,0)--(15,0);
  \end{braid}.
\]
Given our assumptions in this case, it is easy to see that the condition $t+1 \in \Arm(\si^\la_\mu)$ is equivalent to $y < \mu_{x+1}$. If this holds, then Theorem~\ref{thm:KLRonBasis} says the above element is equal to 
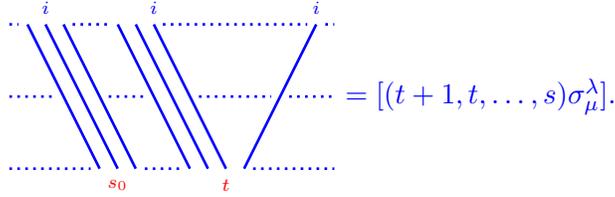
\[
  \begin{braid}\tikzset{scale=0.6,baseline=12mm}
    \draw[dotted](0,8)--(0.5,8);
    \draw[dotted](0,4)--(2.5,4);
    \draw[dotted](0,0)--(4.5,0);
    \draw(1,8)--(5,0);
    \draw(2,8)node[above]{$i$}--(6,0)node[below,color=red]{$s_0$};
    \draw(3,8)--(7,0);
    \draw[dotted](3.5,8)--(5.5,8);
    \draw[dotted](5.5,4)--(7.5,4);
    \draw[dotted](7.5,0)--(9.5,0);
    \draw(6,8)--(10,0);
    \draw(7,8)--(11,0);
    \draw(8,8)node[above]{$i$}--(12,0)node[below,color=red]{$t$};
    \draw[dotted](8.5,8)--(16.5,8);
    \draw[dotted](10.5,4)--(14.5,4);
    \draw(17,8)node[above]{$i$}--(13,0);
    \draw[dotted](17.5,8)--(18,8);
    \draw[dotted](15.5,4)--(18,4);
    \draw[dotted](13.5,0)--(18,0);
  \end{braid}
  \begin{braid}\tikzset{scale=0.6,baseline=12mm}
    \draw(0,4)node[font=\normalsize]{$=[(t+1, t, \dots, s) \si^\la_\mu]$.};
  \end{braid}
\]
as claimed. Otherwise, we get zero.

  We next consider the case that $A =(x,y)$ with $1 \leq x \leq k$ and $y = 1$. If $x = 1$, then $\psi^{\T^A} [\si^\la_\mu] = 0$ by the Garnir relation for the node $(1,1) \in \la$. Otherwise $A$ is a leg node, and so is $(x+1, 1)$. Therefore, $\psi^{\T^A} [\si^\la_\mu]$ takes the following form.
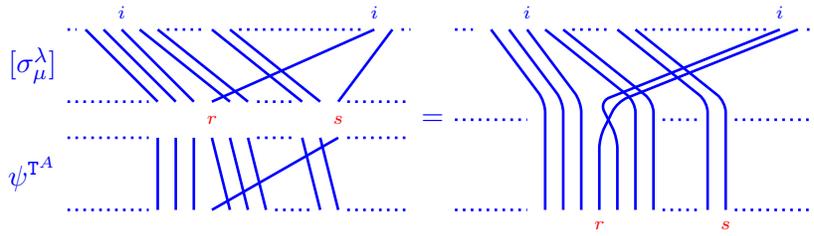
\[
  \begin{braid}\tikzset{scale=0.6,baseline=12mm}
    \draw(-2,8)node[left,font=\normalsize]{$[\si^\la_\mu]$};
    \draw(-2,2)node[left,font=\normalsize]{$\psi^{\T^A}$};

    \draw[dotted](-2,10)--(-1.5,10);
    \draw[dotted](-2,6)--(2.5,6);
    \draw[dotted](-2,4)--(2.5,4);
    \draw[dotted](-2,0)--(2.5,0);

    \draw(-1,10)--(3,6); \draw(3,4)--(3,0);
    \draw(0,10)--(4,6); \draw(4,4)--(4,0);
    \draw(1,10)node[above]{$i$}--(5,6); \draw(5,4)--(5,0);
    \draw(2,10)--(7,6); \draw(7,4)--(8,0);
    \draw(3,10)--(8,6); \draw(8,4)--(9,0);

    \draw[dotted](3.5,10)--(5.5,10);
    \draw[dotted](8.5,6)--(10.5,6);
    \draw[dotted](8.5,4)--(10.5,4);
    \draw[dotted](9.5,0)--(11.5,0);

    \draw(6,10)--(11,6); \draw(11,4)--(12,0);
    \draw(7,10)--(12,6); \draw(12,4)--(13,0);

    \draw[dotted](7.5,10)--(14.5,10);

    \draw(15,10)node[above]{$i$}--(6,6); \draw(6,5)node[color=red]{$r$}; \draw(6,4)--(7,0);
    \draw(16,10)--(13,6); \draw(13,5)node[color=red]{$s$}; \draw(13,4)--(6,0);

    \draw[dotted](16.5,10)--(17,10);
    \draw[dotted](13.5,6)--(17,6);
    \draw[dotted](13.5,4)--(17,4);
    \draw[dotted](13.5,0)--(17,0);
  \end{braid}
  \begin{braid}\tikzset{scale=0.6,baseline=12mm}
    \draw(0,5)node[font=\normalsize]{$=$};
  \end{braid}
  \begin{braid}\tikzset{scale=0.6,baseline=12mm}
    \draw[dotted](-1,10)--(0,10);
    \draw[dotted](-1,5)--(3,5);
    \draw[dotted](-1,0)--(3,0);

    \draw(1,10)--(4,6)--(4,4)--(4,0);
    \draw(2,10)--(5,6)--(5,4)--(5,0);
    \draw(3,10)node[above]{$i$}--(6,6)--(6,4)--(6,0);
    \draw(4,10)--(9,6)--(9,4)--(9,0);
    \draw(5,10)--(10,6)--(10,4)--(10,0);

    \draw[dotted](5.5,10)--(7.5,10);
    \draw[dotted](10.5,5)--(12.5,5);
    \draw[dotted](10.5,0)--(12.5,0);

    \draw(8,10)--(13,6)--(13,4)--(13,0);
    \draw(9,10)--(14,6)--(14,4)--(14,0)node[below, color=red]{$s$};

    \draw[dotted](9.5,10)--(16.5,10);

    \draw(17,10)node[above]{$i$}--(7,6)--(8,4)--(8,0);
%      \draw[color=red](7,6)--(8,4);
%      \draw(8,4)--(8,0);
    \draw(18,10)--(8,6)--(7,4)--(7,0)node[below, color=red]{$r$};
%      \draw[color=red](8,6)--(7,4);
%      \draw(7,4)--(7,0)node[below, color=red]{$r$};

    \draw[dotted](18.5,10)--(19,10);
    \draw[dotted](14.5,5)--(19,5);
    \draw[dotted](14.5,0)--(19,0);
  \end{braid}
\]
which is zero by Theorem~\ref{thm:KLRonBasis}(iii).

  Suppose now that $A =(x,y)$ with $1 \leq x \leq k$, $y \equiv 1$ (mod $e$), and $y \neq 1$. Here, $\psi^{\T^A} [\si^\la_\mu]$ takes the following form.
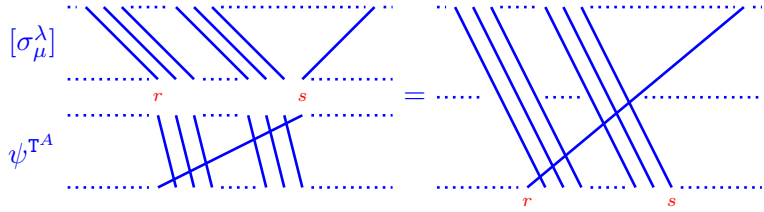
\[
  \begin{braid}\tikzset{scale=0.6,baseline=12mm}
    \draw(0,8)node[left,font=\normalsize]{$[\si^\la_\mu]$};
    \draw(0,2)node[left,font=\normalsize]{$\psi^{\T^A}$};

    \draw[dotted](0,10)--(0.5,10);
    \draw[dotted](0,6)--(4.5,6);
    \draw[dotted](0,4)--(4.5,4);
    \draw[dotted](0,0)--(4.5,0);

    \draw(1,10)--(5,6); \draw(5,5)node[color=red]{$r$}; \draw(5,4)--(6,0);
    \draw(2,10)--(6,6); \draw(6,4)--(7,0);
    \draw(3,10)--(7,6); \draw(7,4)--(8,0);

    \draw[dotted](3.5,10)--(5.5,10);
    \draw[dotted](7.5,6)--(9.5,6);
    \draw[dotted](7.5,4)--(9.5,4);
    \draw[dotted](8.5,0)--(10.5,0);

    \draw(6,10)--(10,6); \draw(10,4)--(11,0);
    \draw(7,10)--(11,6); \draw(11,4)--(12,0);
    \draw(8,10)--(12,6); \draw(12,4)--(13,0);

    \draw[dotted](8.5,10)--(16.5,10);

    \draw(17,10)--(13,6); \draw(13,5)node[color=red]{$s$}; \draw(13,4)--(5,0);

    \draw[dotted](17.5,10)--(18,10);
    \draw[dotted](13.5,6)--(18,6);
    \draw[dotted](13.5,4)--(18,4);
    \draw[dotted](13.5,0)--(18,0);
  \end{braid}
  \begin{braid}\tikzset{scale=0.6,baseline=12mm}
    \draw(0,5)node[font=\normalsize]{$=$};
  \end{braid}
  \begin{braid}\tikzset{scale=0.6,baseline=12mm}
    \draw[dotted](0,10)--(0.5,10);
    \draw[dotted](0,5)--(2.5,5);
    \draw[dotted](0,0)--(4.5,0);

    \draw(1,10)--(6,0);
    \draw(2,10)--(7,0);
    \draw(3,10)--(8,0);

    \draw[dotted](3.5,10)--(5.5,10);
    \draw[dotted](6,5)--(8,5);
    \draw[dotted](8.5,0)--(10.5,0);

    \draw(6,10)--(11,0);
    \draw(7,10)--(12,0);
    \draw(8,10)--(13,0)node[below,color=red]{$s$};

    \draw[dotted](8.5,10)--(16.5,10);

    \draw(17,10)--(5,0)node[below,color=red]{$r$};

    \draw[dotted](17.5,10)--(18,10);
    \draw[dotted](11.5,5)--(18,5);
    \draw[dotted](13.5,0)--(18,0);
  \end{braid}
\]
which is $[(r, r+1, \dots, s) \si^\la_\mu]$.

  Consider now the case $A = (x, y) \in \mu$ is a Garnir node with $1 \leq x \leq k$, and $y \not \equiv 0,1$ (mod $e$). If $\mu$ is as in Lemma~\ref{lem:muForms}(i) or (ii) then the requirement that $A$ is Garnir and $y \not \equiv 1$ (mod $e$) together imply that $x \leq c$. Therefore $\mu_x = a_x e$, and it is easy to deduce from this that $C^A$ and $D^A$ each contain at least two nodes. 
The portion of $\psi^{\T^A} [\si^\la_\mu]$ showing the crossing between $C^A$ and $D^A$ is depicted in the following picture
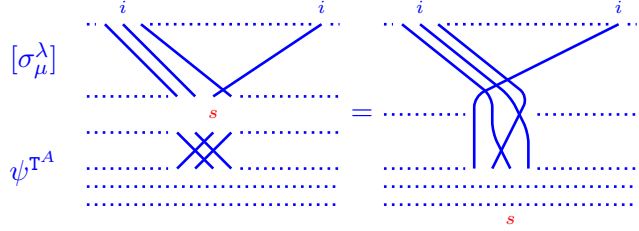
\[
  \begin{braid}\tikzset{scale=0.6,baseline=12mm}
    \draw[dotted](0,10)--(0.5,10);
    \draw[dotted](0,6)--(4.5,6);
    \draw[dotted](0,4)--(4.5,4);
    \draw[dotted](0,2)--(4.5,2);

    \draw(1,10)--(5,6);
      \draw(5,4)--(7,2);
    \draw(2,10)node[above]{$i$}--(6,6);
      \draw(6,4)--(8,2);
    \draw(3,10)--(8,6);
      \draw(8,4)--(6,2);

    \draw[dotted](3.5,10)--(12.5,10);

    \draw(13,10)node[above]{$i$}--(7,6);
      \draw(7,5)node[color=red]{$s$};
      \draw(7,4)--(5,2);

    \draw[dotted](13.5,10)--(14,10);
    \draw[dotted](8.5,6)--(14,6);
    \draw[dotted](8.5,4)--(14,4);
    \draw[dotted](8.5,2)--(14,2);

    \draw[dotted](0,1)--(14,1);
    \draw[dotted](0,0)--(14,0);

    \draw(-1,8)node[left,font=\normalsize]{$[\si^\la_\mu]$};
    \draw(-1,2)node[left,font=\normalsize]{$\psi^{\T^A}$};
  \end{braid}
  \begin{braid}\tikzset{scale=0.6,baseline=12mm}
    \draw(0,5)node[font=\normalsize]{$=$};
  \end{braid}
  \begin{braid}\tikzset{scale=0.6,baseline=12mm}
    \draw[dotted](0,10)--(0.5,10);
    \draw[dotted](0,5)--(4.5,5);
    \draw[dotted](0,2)--(4.5,2);

    \draw(1,10)--(6,6)--(6,4)--(7,2);
       \draw(7,0)node[below,color=red]{$s$};
    \draw(2,10)node[above]{$i$}--(7,6)--(8,4)--(8,2);
%       \draw[color=red](7,6)--(8,4);
%       \draw(8,4)--(8,2);
    \draw(3,10)--(8,6)--(7,4)--(6,2);
%       \draw[color=red](8,6)--(7,4);
%       \draw(7,4)--(6,2);

    \draw[dotted](3.5,10)--(12.5,10);

    \draw(13,10)node[above]{$i$}--(5,6)--(5,4)--(5,2);

    \draw[dotted](13.5,10)--(14,10);
    \draw[dotted](8.5,5)--(14,5);
    \draw[dotted](8.5,2)--(14,2);

    \draw[dotted](0,1)--(14,1);
    \draw[dotted](0,0)--(14,0);
  \end{braid}
\]
which is zero by Theorem~\ref{thm:KLRonBasis}(iii).

  We next check the special case that $A = (k+1, 1)$. If $\mu$ is as in Lemma~\ref{lem:muForms}(ii), then the picture we have is as follows.
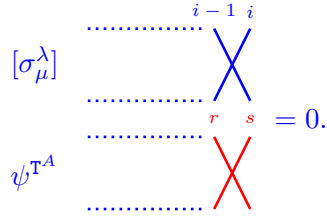
\[
  \begin{braid}\tikzset{scale=0.6,baseline=12mm}
    \draw[dotted](0,10)--(6.5,10);
    \draw[dotted](0,6)--(6.5,6);
    \draw[dotted](0,4)--(6.5,4);
    \draw[dotted](0,0)--(6.5,0);

    \draw(7,10)node[above]{$i-1$}--(9,6);
      \draw[color=red](9,5)node{$s$};
      \draw[color=red](9,4)--(7,0);
    \draw(9,10)node[above]{$i$}--(7,6);
      \draw[color=red](7,5)node{$r$};
      \draw[color=red](7,4)--(9,0);

    \draw(-1,8)node[left,font=\normalsize]{$[\si^\la_\mu]$};
    \draw(-1,2)node[left,font=\normalsize]{$\psi^{\T^A}$};
  \end{braid}
  \begin{braid}\tikzset{scale=0.6,baseline=12mm}
    \draw(0,5)node[font=\normalsize]{$=0$.};
  \end{braid}
\]
If $\mu$ is of the third type, then the picture is as follows
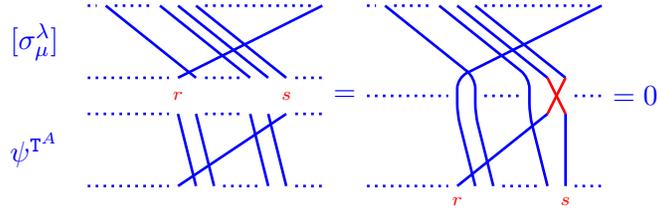
\[
  \begin{braid}\tikzset{scale=0.6,baseline=12mm}
    \draw[dotted](0,10)--(0.5,10);
    \draw[dotted](0,6)--(4.5,6);
    \draw[dotted](0,4)--(4.5,4);
    \draw[dotted](0,0)--(4.5,0);

    \draw(1,10)--(6,6);
      \draw(6,4)--(7,0);

    \draw[dotted](1.5,10)--(3.5,10);
    \draw[dotted](6.5,6)--(8.5,6);
    \draw[dotted](6.5,4)--(8.5,4);
    \draw[dotted](7.5,0)--(9.5,0);

    \draw(4,10)--(9,6);
      \draw(9,4)--(10,0);
    \draw(5,10)--(10,6);
      \draw(10,4)--(11,0);
    \draw(6,10)--(11,6);
      \draw(11,5)node[color=red]{$s$};
      \draw(11,4)--(5,0);

    \draw[dotted](6.5,10)--(12.5,10);

    \draw(13,10)--(5,6);
      \draw(5,5)node[color=red]{$r$};
      \draw(5,4)--(6,0);

    \draw[dotted](11.5,6)--(13,6);
    \draw[dotted](11.5,4)--(13,4);
    \draw[dotted](11.5,0)--(13,0);

    \draw(-1,8)node[left,font=\normalsize]{$[\si^\la_\mu]$};
    \draw(-1,2)node[left,font=\normalsize]{$\psi^{\T^A}$};
  \end{braid}
  \begin{braid}\tikzset{scale=0.6,baseline=12mm}
    \draw(0,5)node[font=\normalsize]{$=$};
  \end{braid}
  \begin{braid}\tikzset{scale=0.6,baseline=12mm}
    \draw[dotted](0,10)--(0.5,10);
    \draw[dotted](0,5)--(4.5,5);
    \draw[dotted](0,0)--(4.5,0);

    \draw(1,10)--(6,6)--(6,4)--(7,0);

    \draw[dotted](1.5,10)--(3.5,10);
    \draw[dotted](6.5,5)--(8.5,5);
    \draw[dotted](7.5,0)--(9.5,0);

    \draw(4,10)--(9,6)--(9,4)--(10,0);
    \draw(5,10)--(10,6);
      \draw[color=red](10,6)--(11,4);
      \draw(11,4)--(11,0)node[below, color=red]{$s$};
    \draw(6,10)--(11,6);
      \draw[color=red](11,6)--(10,4);
      \draw(10,4)--(5,0)node[below, color=red]{$r$};

    \draw[dotted](6.5,10)--(12.5,10);

    \draw(13,10)--(5,6)--(5,4)--(6,0);

    \draw[dotted](11.5,5)--(13,5);
    \draw[dotted](11.5,0)--(13,0);
  \end{braid}
  \begin{braid}\tikzset{scale=0.6,baseline=12mm}
    \draw(0,5)node[font=\normalsize]{$=0$.};
  \end{braid}
\]

  Suppose now that $A = (x,y)$ with $x > k$ and $y \not \equiv 0$ (mod $e$). Furthermore assume that $A \neq (k+1, 1)$. We see that $C^A$ is nonempty, and so $\psi^{\T^A} [\si^\la_\mu]$ is as in the following picture
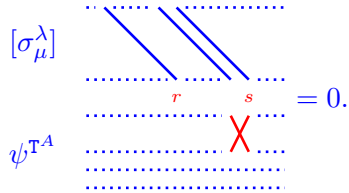
\[
  \begin{braid}\tikzset{scale=0.6,baseline=12mm}
    \draw[dotted](0,10)--(0.5,10);
    \draw[dotted](0,6)--(4.5,6);
    \draw[dotted](0,4)--(7.5,4);
    \draw[dotted](0,2)--(7.5,2);

    \draw(1,10)--(5,6);
      \draw(5,5)node[color=red]{$r$};

    \draw[dotted](1.5,10)--(3.5,10);
    \draw[dotted](5.5,6)--(7.5,6);

    \draw(4,10)--(8,6);
      \draw[color=red](8,4)--(9,2);
    \draw(5,10)--(9,6);
      \draw(9,5)node[color=red]{$s$};
      \draw[color=red](9,4)--(8,2);

    \draw[dotted](5.5,10)--(11,10);
    \draw[dotted](9.5,6)--(11,6);
    \draw[dotted](9.5,4)--(11,4);
    \draw[dotted](9.5,2)--(11,2);

    \draw[dotted](0,1)--(11,1);
    \draw[dotted](0,0)--(11,0);

    \draw(-1,8)node[left,font=\normalsize]{$[\si^\la_\mu]$};
    \draw(-1,2)node[left,font=\normalsize]{$\psi^{\T^A}$};
  \end{braid}
  \begin{braid}\tikzset{scale=0.6,baseline=12mm}
    \draw(0,5)node[font=\normalsize]{$=0$.};
  \end{braid}
\]

  Finally, if we choose $A = (x,y)$ with $x > k$ and $y \equiv 0$ (mod $e$) one may easily see that $\T^A = \T^\mu$. Clearly then $\psi^{\T^A}=1$, and so $\psi^{\T^A} [\si^\la_\mu] = [\si^\la_\mu]$.
\end{proof}

\begin{Lemma}\label{lem:GarnirRel}
  Suppose that $\mu$ is as in Lemma~\ref{lem:muForms} and let $A \in \mu$ be a Garnir node. Let $r, s, t$ be as in Lemma~\ref{lem:GarnirPsi}. Then for $1 \leq f \leq a_{x+1}-1$ we have
  \[
    g^A [\si^\la_\mu] = \begin{cases}
      \binom{a_x}{f} [(t+1, t, \dots, s) \si^\la_\mu],& \textup{if $A = (x, fe)$ with $x \leq k$} \\
      \binom{a_x}{f} [(r, r+1, \dots, s) \si^\la_\mu],& \textup{if $A = (x, fe + 1)$ with $x \leq k$}\\
      \binom{a_x}{f} [\si^\la_\mu],& \textup{if $A = (x, fe)$ with $x > k$}\\
      0,& \textup{otherwise}.
    \end{cases}
  \]
\end{Lemma}

\begin{proof}
  If $\psi^{\T^A} [\si^\la_\mu] = 0$ there is nothing to show, so assume otherwise.

  Recall the definitions of $\si^A_p$ and $\tau^A_p$ from Section~\ref{SecSpecht}. In each of the three cases in Lemma~\ref{lem:GarnirPsi} for which $\psi^{\T^A} [\si^\la_\mu] \neq 0$, one sees that multiplying $\si^A_p$ with $\psi^{\T^A} [\si^\la_\mu]$ has the effect of shuffling $e$ arm strands past $e$ arm strands. By Theorem~\ref{thm:KLRonBasis}, we have that $\si^A_p \psi^{\T^A} [\si^\la_\mu] = 0$. Thus $\tau^A_p \psi^{\T^A} [\si^\la_\mu] = \psi^{\T^A} [\si^\la_\mu]$ for all $p$. This furthermore implies that $\tau^A_w \psi^{\T^A} [\si^\la_\mu] = \psi^{\T^A} [\si^\la_\mu]$ for all $w \in \Si^A$. Therefore
\[
  g^A [\si^\la_\mu] = \sum_{u \in \D^A} \tau^A_u \psi^{\T^A} [\si^\la_\mu]
             = \sum_{u \in \D^A} \psi^{\T^A} [\si^\la_\mu]
             = \left| \D^A \right| \psi^{\T^A} [\si^\la_\mu].
\]

The problem is reduced to calculating $\left| \D^A \right|$ in each of the three cases above, which follows immediately from the definitions.
\end{proof}

%We are now able to prove the main theorem.

\emph{Proof of Theorem~\ref{thm:mainHom}}
  By the discussion preceding Lemma~\ref{lem:muForms}, any homomorphism $S^\mu \to S^\la$ is of the form $z^\mu \mapsto \ga [\si^\la_\mu]$ for some $\ga \in \O$. By the discussion preceding Lemma~\ref{lem:muForms} if there is a nonzero homomorphism, then $[\si_{\mu}] = [\si^\la_\mu]$. Lemma~\ref{lem:muForms} tells us for which partitions this happens. 

  Suppose now that $\mu$ is as in Lemma~\ref{lem:muForms}. Corollary~\ref{cor:muForms} ensures that $(J^\mu_1 + J^\mu_2 + J^\mu_3) \ga [\si^\la_\mu] = 0$. By Lemma~\ref{lem:GarnirRel}, $J^\mu_4 \ga [\si^\la_\mu] = 0$ if and only if for each $x = 1, \dots, N-1$, we have $\binom{a_x}{f} \ga = 0$ for $f = 1, \dots, a_{x+1} - 1$. But $\Gc(\ba)$ is the greatest common divisor of these binomial coefficients. Therefore $\ga [\si^\la_\mu]$ satisfies the Garnir relations if and only if $\Gc(\ba) \ga = 0$. This is exactly the claim of the theorem.
\qed

\section{Examples}

We now present some examples of Theorem~\ref{thm:mainHom}. We will restrict our attention to the case that $\O$ is a torsion-free algebra over either $\Z$ or $\F_p$ for some prime $p \geq 3$. In this case for any $r \in \Z$ the submodule $\Ann_\O(r)$ is either $(0)$ or $\O$, and so in Theorem~\ref{thm:mainHom} all of our homomorphism spaces will be free $\O$-modules. In particular, we can talk about their dimensions.

In characteristic $0$ the situation simplifies. The following corollary covers in particular Hecke algebras over $\C$ at a root of unity. %We need to treat the trivial module separately.

\begin{Corollary}
  Let $\mu$ be an arbitrary partition, and $\la = (d-k, 1^k)$. Then $\Hom(S^\mu, S^\la)$ is at most one-dimensional, and $\dim \Hom(S^\mu, S^\la) = 1$ if and only if one of the following holds:
  \begin{enumerate}
    \item $\mu = \la$;
    \item $e$ divides $d$ and $\mu = (d-k-1, 1^{k+1})$;
    \item $k = 0$ and there exist $n \geq 0$, $a > 0$, and $0 \leq m < e$ such that 
$$\mu = (ae-1, (e-1)^n, m);$$
    \item $k \geq 1$, there exist $0 \leq n < k$, $a \in \Z_{>0}$, and $1 \leq m \leq e$ such that 
$$\mu = (a e, e^n, m, 1^{k-n-1});$$
    \item $k \geq 1$, there exist $n > k+1$, $a \in \Z_{>0}$ and $1 \leq m < e$ such that 
$$\mu = (a e, e^{k-1}, (e - 1)^{n-k-1}, m);$$
    \item $k \geq 2$, $e$ divides $d$, there exist $0 \leq n \leq k-2$, $a \in \Z_{>0}$, and $1 \leq m \leq e$ such that 
$$\mu = (a e, e^n, m, 1^{k-n}).$$
  \end{enumerate}
\end{Corollary}

For the rest of the section we assume that $\O$ is an $\F_p$-algebra, with $p \geq 3$, and we furthermore assume that $e=p$. This applies in the case of symmetric groups in positive characteristic.

\subsection{The trivial module}

Let $\la = (d)$. The module $S^{\la}$ is referred to as the \emph{trivial} module. It is one-dimensional over $\O$ having basis $z^\la$. 
We recover the following graded version of the result of James \cite[Theorem 24.4]{J}.

\begin{Corollary}\label{cor:triv}
  Let $\mu = (\mu_1, \dots, \mu_N)$ be a partition of $d$. Then $\Hom_{R_\al}(S^\mu, S^\la) = 0$ unless there exist $a_1, \dots, a_N \in \Z_{>0}$ and $m \in \{0, \dots, p-1\}$ such that $\mu = (a_1 p - 1, \dots, a_{N-1} p - 1, a_N p - 1 - m)$, and furthermore $\nu_p(a_x) \geq \ell_p(a_{x+1} - 1)$ for $1 \leq x < N$ . In this case, 
  $$\qdim \Hom_{R_\al}(S^\mu, S^\la) = q^r \textup{ where } r=N-\left\lceil \frac{N+m}{p} \right\rceil.$$
\end{Corollary}

\begin{proof}
  Theorem~\ref{thm:mainHom}(i) can only occur if $\mu = \la$, and Theorem~\ref{thm:mainHom}(ii) never occurs because there are no valid choices for $c$. Therefore Theorem~\ref{thm:mainHom}(iii) and Corollary~\ref{cor:pGc} tell us the form that $\mu$ must take.

  To calculate the degree of the homomorphisms, we calculate the degrees of $\T^\la$ and $\T^\mu$. For any partition $\nu$ with $M$ parts, we have $\deg(\T^\nu) = \sum_{x=1}^M \lfloor \frac{\nu_x}{p} \rfloor$. In particular we get $\deg(\T^\mu) = \sum_{x=1}^N a_x - N$ and $\deg(\T^\la) = \sum_{x=1}^N a_x - \lceil \frac{N+m}{p} \rceil$, and
  $$\deg(T^\la) - \deg(\T^\mu) = N-\left\lceil \frac{N+m}{p} \right\rceil.$$
\end{proof}

\subsection{The standard module}

Let $\la = (d-1, 1)$. Then $S^\la$ is referred to as the \emph{standard} or \emph{reflection} module. We have the following Corollary.

\begin{Corollary}
For any partition $\mu \neq \la$, $\dim \Hom_{R_\al}(S^\mu, S^\la) = 0$ with the following exceptions:
\begin{enumerate}
  \item if there are $a_1, a_2 \in \Z_{>0}$ and $0 \leq m < p$ such that
    $$\mu = (a_1 p, a_2 p - m),$$
  and futhermore $\nu_p(a_1) \geq \ell_p(a_2-1)$, then
    $$\qdim \Hom_{R_\al}(S^\mu, S^\la) = \begin{cases} 1, &\text{if } p \mid d-1\\
        q, &\text{otherwise;} \end{cases}$$
  \item if $p$ divides $d$ and $\mu = (d-2, 1, 1)$, then $\qdim \Hom_{R_\al}(S^\mu, S^\la) = q$;
  \item if there are $0 \leq m < p$, $N \geq 3$, and $a_1, \dots, a_N \in \Z_{>0}$ such that
    $$\mu = (a_1p, a_2p-1, a_3p-1, \dots, a_{N-1}p-1,a_Np-1-m),$$
  and furthermore $\nu_p(a_i) \geq \ell_p(a_{i+1}-1)$ for $i = 1,\dots,N-1$, then
  $\qdim \Hom_{R_\al}(S^\mu, S^\la) = q^r$ where
  $$r = N-\left\lceil \frac{N+m}{p} \right\rceil + \begin{cases} -1, &\text{if } p \mid d-1\\
        1, &\text{if } p \mid d\\
        0, &\text{otherwise.} \end{cases}$$
\end{enumerate}
\end{Corollary}

\begin{proof}
Everything follows directly from Theorem~\ref{thm:mainHom} and Corollary~\ref{cor:pGc} except the degree calculations. The degree in cases (i) and (ii) is straightforward to calculate, and for case (iii) we can make the following calculations. 
\begin{align*}
\deg(\T^\mu) &= \sum_{x=1}^N a_x - N + 1 \\
\deg(\T^\la) &= \sum_{x=1}^N a_x - \left\lceil \frac{N+m}{p} \right\rceil \\
\deg([\si_\mu]) &= \deg(\T^\la) + \begin{cases} 0, &\text{if } p \mid d-1\\
        2, &\text{if } p \mid d\\
        1, &\text{otherwise.} \end{cases}
\end{align*}
\end{proof}

\end{document}